\documentclass[10pt]{article}
\usepackage{amssymb,enumerate}
\usepackage{amsmath,amsfonts}
\usepackage{amsthm}
\usepackage{latexsym}
\usepackage{mathrsfs}
\usepackage{color}
\usepackage{microtype}
\usepackage{geometry}

\usepackage{algorithm}
\usepackage{algorithmicx}
\usepackage{algpseudocode}
\usepackage{xcolor}
\usepackage{multirow}

\allowdisplaybreaks

\DeclareMathOperator{\cO}{\ensuremath{\mathcal{O}}}

\DeclareMathOperator{\bR}{\ensuremath{\mathbb{R}}}

\DeclareMathOperator{\sgn}{\ensuremath{\mathrm{sgn}}}

\newtheorem{lemma}{Lemma}

\newtheorem{assumption}{Assumption}

\newtheorem{theorem}{Theorem}
\newtheorem{remark}{Remark}

\newcommand{\beq}{\begin{equation}}
\newcommand{\eeq}{\end{equation}}
\newcommand{\beqa}{\begin{eqnarray}}
\newcommand{\eeqa}{\end{eqnarray}}
\newcommand{\beqas}{\begin{eqnarray*}}
\newcommand{\eeqas}{\end{eqnarray*}}
\newcommand{\ba}{\begin{array}}
\newcommand{\ea}{\end{array}}
\newcommand{\bi}{\begin{statementize}}
\newcommand{\ei}{\end{statementize}}

\def\fl{{f_{\rm low}}}

\def\nn{{\nonumber}}

\def\scL{{\mathscr L}}

\title{Newton-CG methods for nonconvex unconstrained optimization with H\"older continuous Hessian}

\author{
Chuan He\thanks{Department of Mathematics, Link\"oping University, Sweden (email: {\tt chuan.he@liu.se}).}
\and
Heng Huang\thanks{Department of Computer Science, University of Maryland, USA (email: {\tt heng@umd.edu}).}
\and
Zhaosong Lu\thanks{Department of Industrial and Systems Engineering, University of Minnesota, USA (email: {\tt zhaosong@umn.edu}).}
}

\date{November 13, 2023 (Revised: December 10, 2024; April 13, 2025)}

\begin{document}
\maketitle
\begin{abstract}
In this paper we consider a nonconvex unconstrained optimization problem minimizing a twice differentiable objective function with H\"older continuous Hessian. Specifically, we first propose a Newton-conjugate gradient (Newton-CG) method for finding an approximate first- and second-order stationary point of this problem, assuming the associated the H\"older parameters are explicitly known. Then we develop a parameter-free Newton-CG method without requiring any prior knowledge of these parameters. To the best of our knowledge, this method is the first parameter-free second-order method achieving the best-known iteration and operation complexity for finding an approximate first- and second-order stationary point of this problem.  Finally, we present preliminary numerical results to demonstrate the superior practical performance of our parameter-free Newton-CG method over a well-known regularized Newton method.
\end{abstract}


\bigskip


\section{Introduction}
In this paper we consider the nonconvex unconstrained optimization problem
\begin{equation}\label{ucpb}
\min_{x\in\bR^n} f(x),
\end{equation}
where $f:\bR^n\to\bR$ is twice continuously differentiable and $\nabla^2 f$ is H\"older continuous in an open neighborhood of a level set of $f$ (see Assumption~\ref{asp:NCG-cmplxity} for details). Our goal is to propose \emph{easily implementable} second-order methods with complexity guarantees,  particularly, Newton-conjugate gradient (Newton-CG) methods for finding approximate first- and second-order stationary points of problem \eqref{ucpb}. 

In recent years, there have been significant advancements in second-order methods with complexity guarantees for problem~\eqref{ucpb} when $\nabla^2 f$ is \emph{Lipschitz continuous}. Notably, cubic regularized Newton methods~\cite{AgAl17local,CaDu19,CaGoTo11,NePo06cubic}, trust-region methods~\cite{CuRoRoWr21,CuRoSa17,MaRa17}, second-order line-search method~\cite{RoWr18}, inexact regularized Newton method~\cite{CuRoSa19}, quadratic regularization method~\cite{BiMa17use}, and Newton-CG method~\cite{RNW18} were developed for finding an $(\epsilon,\sqrt{\epsilon})$-second-order stationary point (SOSP) $x$ of problem~\eqref{ucpb} satisfying
\[
\|\nabla f(x)\|\le\epsilon,\qquad \lambda_{\min}(\nabla^2 f(x))\ge-\sqrt{\epsilon},
\]
where $\epsilon\in(0,1)$ is a tolerance parameter and $\lambda_{\min}(\cdot)$ denotes the minimum eigenvalue of the associated matrix. Under suitable assumptions, it was shown that these second-order methods achieve an iteration complexity of $\cO(\epsilon^{-3/2})$ for finding an $(\epsilon,\sqrt{\epsilon})$-SOSP, which has been proved to be optimal in \cite{CaDuHiSi20,CaGoTo18}. In addition to iteration complexity, operation complexity of the methods in \cite{AgAl17local,CaDu19,CuRoRoWr21,RNW18,RoWr18}, measured by the number of their fundamental operations, was also studied. Under suitable assumptions, it was shown that these methods achieve an operation complexity of $\widetilde{\cO}(\epsilon^{-7/4})$ for finding an $(\epsilon,\sqrt{\epsilon})$-SOSP of problem~\eqref{ucpb} with high probability.\footnote{$\widetilde{\cO}(\cdot)$ represents $\cO(\cdot)$ with logarithmic terms omitted.} Similar operation complexity bounds have also been achieved by gradient-based methods (e.g., see \cite{AlLi18,CaDuHiSi17,CaDuHiSi18,JiNeJo18,LiLi22,MaTa22,XuJiYa17}).

Nonetheless, there has been limited study on second-order methods for problem \eqref{ucpb} -- a nonconvex unconstrained optimization problem with H\"older continuous Hessian. The regularized Newton methods proposed in \cite{cartis2011adaptive,cartis2019universal,cartis2020sharp,GeNe17NH, zhang2023riemannian} appear to be the only existing second-order methods for problem \eqref{ucpb}. Specifically, the cubic regularized Newton method in \cite{GeNe17NH} tackles problem \eqref{ucpb} by solving a sequence of cubic regularized Newton subproblems. It is a parameter-free second-order method and does not require any prior information on the modulus $H_\nu$ and exponent $\nu$  associated with the H\"older continuity (see Assumption~\ref{asp:NCG-cmplxity}). Under mild assumptions, it was shown in \cite{GeNe17NH} that this method enjoys an iteration complexity of
\begin{equation}\label{opt-2nd-ic}
\cO(H_\nu^{1/(1+\nu)}\epsilon_g^{-(2+\nu)/(1+\nu)})
\end{equation}
for finding an $\epsilon_g$-first-order stationary point (FOSP) $x$ of problem \eqref{ucpb} satisfying $\|\nabla f(x)\|\le\epsilon_g$. This iteration complexity matches the lower iteration complexity bound established in \cite{cartis2011optimal,CaGoTo18}. In another early work \cite{cartis2011adaptive}, a regularized Newton method, which solves a sequence of $(2+\nu)$th-order regularized Newton subproblems, has been proposed for solving problem \eqref{ucpb}. Iteration complexity of this method for finding an $\epsilon_g$-FOSP has been established, with the order  dependence on $\epsilon_g$ matching the optimal one in \eqref{opt-2nd-ic}. This method has been generalized in \cite{cartis2019universal,cartis2020sharp} to regularized high-order methods for solving nonconvex problems with H\"older continuous high-order derivatives. It shall be noted that  these methods \cite{cartis2011adaptive,cartis2019universal,cartis2020sharp,GeNe17NH} require solving regularized Newton or high-order polynomial optimization problems {\it exactly} per iteration, which may be highly expensive to implement in general.
Recently, in \cite{zhang2023riemannian}, two adaptive regularized Newton methods were proposed for finding an approximate SOSP  of the problem minimizing a nonconvex function with H\"older continuous Hessian on a Riemannian manifold, which includes problem \eqref{ucpb} as a special case. Specifically, when applied to problem \eqref{ucpb}, one method in \cite{zhang2023riemannian}  inexactly solves a sequence of $(2+\nu)$th-order regularized Newton subproblems, while another method in \cite{zhang2023riemannian} inexactly solves a sequence of trust-region subproblems. It has been shown in \cite{zhang2023riemannian} that their methods exhibit an iteration complexity of ${\cO}(\epsilon_g^{-(2+\nu)/(1+\nu)})$ for finding an $\epsilon_g$-FOSP of problem \eqref{ucpb}, which achieves the optimal order of dependence on $\epsilon_g$ as given in \eqref{opt-2nd-ic}. However, these methods in \cite{zhang2023riemannian} 
are \emph{not fully parameter-free} since prior knowledge of the H\"older exponent is required in order to achieve the best-known complexity.

As discussed above, the existing second-order methods \cite{cartis2011adaptive,cartis2019universal,cartis2020sharp,GeNe17NH, zhang2023riemannian} for problem \eqref{ucpb} require solving a sequence of sophisticated trust-region or regularized Newton subproblems. In this paper, we propose \emph{easily implementable} second-order methods, particularly Newton-CG methods for \eqref{ucpb}, by applying the capped CG method \cite[Algorithm~1]{RNW18} to solve a sequence of systems of linear equations with coefficient matrix resulting from a proper perturbation on the Hessian of $f$. Specifically, we first propose a Newton-CG method (Algorithm~\ref{alg:NCG-pd}) to find an approximate FOSP and SOSP of \eqref{ucpb}, assuming the parameters associated with the H\"older continuity of $\nabla^2 f$ are explicitly known. Then we develop a {\it parameter-free} Newton-CG method (Algorithm~\ref{alg:NCG}) for finding an approximate FOSP and SOSP of \eqref{ucpb} without requiring any prior knowledge of these parameters. We show that these methods achieve the best-known iteration and operation complexity for finding an approximate FOSP and/or SOSP of \eqref{ucpb}. Moreover, when $\nabla^2 f$ is Lipschitz continuous, our proposed methods achieve an improved iteration and operation complexity over the Newton-CG methods \cite{hlp2023ncgal,RNW18} in terms of the dependence on the Lipschitz constant of $\nabla^2 f$. In addition, preliminary numerical results are presented, demonstrating the practical advantages of our parameter-free Newton-CG method over the cubic regularized Newton method \cite{GeNe17NH}.

The main contributions of this paper are summarized as follows.
\begin{itemize}
\item  We propose a Newton-CG method (Algorithm~\ref{alg:NCG-pd}) to find an approximate FOSP and SOSP of \eqref{ucpb}, assuming that the parameters associated with the H\"older continuity of $\nabla^2 f$ are explicitly known. In contrast with the regularized Newton methods \cite{cartis2011adaptive,GeNe17NH, zhang2023riemannian}, our method is {\it easily implementable} and solves much {\it simpler subproblems} by a capped CG method, while achieving the best-known iteration and operation complexity.

\item We propose a {\it parameter-free} Newton-CG method (Algorithm~\ref{alg:NCG}) for finding an approximate FOSP and SOSP of \eqref{ucpb} without requiring any prior knowledge of these parameters. To the best of our knowledge, this is {\it the first fully parameter-free method} for finding an approximate FOSP and SOSP of \eqref{ucpb}, while achieving the best-known iteration and operation complexity. 

\end{itemize}


The rest of this paper is organized as follows. In Section~\ref{sec:not-as}, we introduce some notation and assumptions that will be used in the paper. In Section~\ref{sec:pd-ncg}, we propose a Newton-CG method for problem \eqref{ucpb} and study its complexity. In Section~\ref{sec:ncg}, we propose a parameter-free Newton-CG method for problem \eqref{ucpb} and study its complexity. Section~\ref{sec:num} presents preliminary numerical results. In Section~\ref{sec:proof}, we present the proofs of the main results. 

\section{Notation and assumptions}\label{sec:not-as}
Throughout this paper, we let $\bR^n$ denote the $n$-dimensional Euclidean space. We use $\|\cdot\|$ to denote the Euclidean norm of a vector or the spectral norm of a matrix. For any $s\in\bR$, we let $s_+$ and $\lceil s\rceil$ denote the nonnegative part of $s$ and the least integer no less than $s$, respectively, and we let $\sgn(s)$ be $1$ if $s\ge0$ and $-1$ otherwise. For a real symmetric matrix $H$, we use $\lambda_{\min}(H)$ to denote its minimum eigenvalue. 
In addition, $\widetilde{\cO}(\cdot)$ represents $\cO(\cdot)$ with logarithmic terms omitted.

We make the following assumptions on problem~\eqref{ucpb} throughout this paper.
\begin{assumption}\label{asp:NCG-cmplxity}
\begin{enumerate}[{\rm (a)}]
\item The level set $\scL_f(x^0):=\{x: f(x)\le f(x^0)\}$ is compact for some $x^0\in \bR^n$.
\item The function $f:\bR^n\to\bR$ is twice continuously differentiable, and $\nabla^2 f$ is H\"older continuous in a bounded convex open neighborhood, denoted by $\Omega$, of $\scL_f(x^0)$, i.e., there exist $\nu\in [0,1]$ and a finite $H_{\nu}>0$ such that
\begin{equation}\label{F-Hess-Lip}
\|\nabla^2 f(x)-\nabla^2 f(y)\|\le H_{\nu}\|x-y\|^{\nu},\quad \forall x,y\in \Omega.
\end{equation}
\end{enumerate}
\end{assumption}

It follows from Assumption~\ref{asp:NCG-cmplxity}(a) that there exist $\fl\in \bR$, $U_g>0$ and $U_H>0$ such that
\begin{equation}\label{lwbd-Hgupbd}
f(x)\ge \fl,\quad \|\nabla f(x)\|\le U_g,\quad \|\nabla^2 f(x)\|\le U_H, \quad \forall x\in\scL_f(x^0).
\end{equation}

We now make some remarks on Assumption~\ref{asp:NCG-cmplxity}(b).
\begin{remark}
\begin{enumerate}[{\rm (i)}]
\item When $\nu=1$, the condition~\eqref{F-Hess-Lip} corresponds to the standard Lipschitz continuity of $\nabla^2 f$. When $\nu=0$, the condition~\eqref{F-Hess-Lip} means that the variation of $\nabla^2 f$ on $\Omega$ is bounded, which is equivalent to the boundedness of $\nabla^2 f$ on $\Omega$. 
\item As a consequence of Assumption~\ref{asp:NCG-cmplxity}(b), the following two inequalities hold for all $x,y\in\Omega$ (e.g., see equations (2.7) and (2.8) in \cite{GeNe17NH}):
\begin{eqnarray}
&&\|\nabla f(y)-\nabla f(x)-\nabla^2 f(x)(y-x)\|\le \frac{H_{\nu}\|y-x\|^{1+\nu}}{1+\nu}, \label{apx-nxt-grad}\\
&&f(y)\le f(x) + \nabla f(x)^T(y-x) + \frac{1}{2}(y-x)^T\nabla^2 f(x) (y-x) + \frac{H_{\nu}\|y-x\|^{2+\nu}}{(1+\nu)(2+\nu)}. \label{desc-ineq}
\end{eqnarray}
\end{enumerate}
\end{remark}

\section{A Newton-CG method for problem~\eqref{ucpb}}\label{sec:pd-ncg}

In this section, we propose a Newton-CG method in  Algorithm~\ref{alg:NCG-pd} for finding  an $\epsilon_g$-FOSP and $(\epsilon_g,\epsilon_H)$-SOSP of problem~\eqref{ucpb}, assuming the parameters $H_\nu$ and $\nu$ associated with the H\"older continuity of $\nabla^2 f$ in \eqref{F-Hess-Lip} are explicitly known, and then analyze its complexity results.


Our Newton-CG method uses two important subroutines,  a capped CG method and a minimum eigenvalue oracle.  Specifically, the capped CG method is a modified CG method  proposed in \cite[Algorithm~1]{RNW18} for solving a possibly indefinite linear system
\begin{equation}\label{indef-sys}
	(H+2\varepsilon I)d = -g,
\end{equation}
where $0\neq g\in\bR^n$, $\varepsilon>0$, and $H\in\bR^{n\times n}$ is a symmetric matrix. It terminates within a finite number of iterations and returns either an approximate solution $d$ to \eqref{indef-sys} satisfying $\|(H+2\varepsilon I)d+g\|\le\hat{\zeta}\|g\|$ and $d^THd\ge -\varepsilon\|d\|^2$ for some $\hat{\zeta}\in(0,1)$ or a sufficiently negative curvature direction $d$ of $H$ with $d^THd<-\varepsilon\|d\|^2$. In addition, the minimum eigenvalue oracle was proposed in \cite[Procedure~2]{RNW18} to check whether a  sufficiently negative curvature direction exists for a symmetric matrix $H$. It either produces a sufficiently negative curvature direction $v$ of $H$ satisfying $\|v\|=1$ and $v^THv\le-\varepsilon/2$ or certifies that $\lambda_{\min}(H)\ge-\varepsilon$ holds with high probability. 
For ease of reference, we present the capped CG method  and the minimum eigenvalue oracle in Algorithms~\ref{alg:capped-CG} and \ref{pro:meo} in Appendices~\ref{appendix:capped-CG} and \ref{appendix:meo}, respectively.

We are now ready to introduce our Newton-CG method (Algorithm~\ref{alg:NCG-pd}) for solving problem~\eqref{ucpb}. This algorithm has two options: (i) when the tolerance $\epsilon_H\in(0,1)$ for second-order stationarity is not provided, this algorithm can find an $\epsilon_g$-FOSP $x$ of \eqref{ucpb} that satisfies $\|\nabla f(x)\|\le\epsilon_g$ for some $\epsilon_g\in(0,1)$; (ii) when such an $\epsilon_H$ is provided, this algorithm can find a stochastic $(\epsilon_g,\epsilon_H)$-SOSP $x$ of \eqref{ucpb}, satisfying $\|\nabla f(x)\|\le\epsilon_g$ deterministically and $\lambda_{\min}(\nabla^2 f(x))\ge-\epsilon_H$ with probability at least $1-\delta$.

Specifically, at each iteration $k$ of Algorithm~\ref{alg:NCG-pd}, if the current iterate $x^k$ does not satisfy $\|\nabla f(x^k)\|\le\epsilon_g$, the capped CG method (Algorithm~\ref{alg:capped-CG}) is invoked to find either an inexact Newton direction or a negative curvature direction by solving the following damped Newton system:
\begin{equation}\label{damp-n}
\big(\nabla^2 f(x^k) + 2(\gamma_\nu(\epsilon_g)\epsilon_g)^{1/2}I\big)d = - \nabla f(x^k),
\end{equation}
where $\gamma_\nu(\epsilon_g)$ is an inexact Lipschitz constant\footnote{In the literature (e.g., \cite{Dv17,ItLuHe23,Ne15}), inexact Lipschitz constant of $\nabla f$ has been used to design and analyze first-order methods for  problem~\eqref{ucpb}, where 
  $f$ has a H\"older continuous gradient.} of $\nabla^2 f$ defined as
\begin{equation}\label{gma-eps}
\gamma_{\nu}(\epsilon_g) := 4 H_{\nu}^{2/(1+\nu)}\epsilon_g^{-(1-\nu)/(1+\nu)}.
\end{equation}
The next iterate $x^{k+1}$ is generated by performing a line search along the descent direction obtained from solving \eqref{damp-n}. Otherwise, if $\|\nabla f(x^k)\|\le\epsilon_g$, this algorithm has two options. First, when the tolerance $\epsilon_H\in(0,1)$ for second-order stationarity is not provided, Algorithm~\ref{alg:NCG-pd} terminates with $x^k$ as an $\epsilon_g$-FOSP. Second, when such an $\epsilon_H$ is provided, a minimum eigenvalue oracle (Algorithm~\ref{pro:meo}) is further invoked to either obtain a sufficiently negative curvature direction and generate the next iterate $x^{k+1}$ via a line search, or certify that $x^k$ is an $(\epsilon_g,\epsilon_H)$-SOSP with high probability and terminates the algorithm. The details of this algorithm are presented in Algorithm~\ref{alg:NCG-pd}.

\begin{algorithm}[!tbh]
{\footnotesize
\caption{A Newton-CG method for problem~\eqref{ucpb}}
\label{alg:NCG-pd}
\begin{algorithmic}
\State {\bf input}: tolerance $\epsilon_g\in(0,1)$, starting point $x^0$, CG-accuracy parameter $\zeta\in(0,1)$, backtracking ratio $\theta\in(0,1)$, line-search parameter $\eta\in(0,1)$, probability parameter $\delta\in(0,1)$, $\gamma_{\nu}(\epsilon_g)$ given in \eqref{gma-eps}; {\bf optional input}: tolerance $\epsilon_H\in(0,1)$;
\For{$k=0,1,2,\ldots$}
\\\hrulefill 
\\ \Comment{{\it This part aims to improve first-order stationarity by calling Algorithm~\ref{alg:capped-CG}.}}
\vspace{1mm}
\If{$\|\nabla f(x^k)\|>\epsilon_g$} 
\State Call Algorithm~\ref{alg:capped-CG} (Appendix~\ref{appendix:capped-CG}) with $H=\nabla^2 f(x^k)$, $\varepsilon=(\gamma_\nu(\epsilon_g)\epsilon_g)^{1/2}$, $g=\nabla f(x^k)$, accuracy parameter $\zeta$, and $U=0$ to
\State obtain outputs $d$, d$\_$type;
\If{d$\_$type=NC}
\State Set
\begin{equation*}
d^k \leftarrow -\sgn(d^T\nabla f(x^k)){\max\{1,1/\gamma_\nu(\epsilon_g)\}}\frac{|d^T\nabla^2 f(x^k) d|}{\|d\|^3}d;
\end{equation*}
\Else\ \{d$\_$type=SOL\}
\State Set
\begin{equation*}
d^k \leftarrow d;
\end{equation*}
\EndIf
\State Go to \textbf{Line Search};
\ElsIf{$\|\nabla f(x^k)\|\le\epsilon_g$ and $\epsilon_H$ is not provided}
\State Output $x^k$ and terminate;
\EndIf
\\\hrulefill
\\ \Comment{{\it This part aims to improve second-order stationarity by calling Algorithm~\ref{pro:meo}.}}
\vspace{1mm}
\If{$\|\nabla f(x^k)\|\le\epsilon_g$ and $\epsilon_H$ is provided}
\State Call Algorithm~\ref{pro:meo} (Appendix~\ref{appendix:meo}) with $H=\nabla^2 f(x^k)$, $\varepsilon=\epsilon_H$, and probability parameter $\delta$;
\If{Algorithm~\ref{pro:meo} certifies that $\lambda_{\min}(\nabla^2 f(x^k))\ge-\epsilon_H$}
\State Output $x^k$ and terminate;
\Else\ \{Sufficiently negative curvature direction $v$ returned by Algorithm~\ref{pro:meo}\}
\State Set d$\_$type=MEO and
\begin{equation}\label{dk-nc-meo2}
    d^k \leftarrow -\mathrm{sgn}(v^T\nabla f(x^k))|v^T\nabla^2 f(x^k) v| v; 
\end{equation}
\State Go to \textbf{Line Search};
\EndIf
\EndIf
\\\hrulefill
\\ \Comment{{\it This part provides line search procedures.}}
\vspace{1mm}
\State \textbf{Line Search:}
\If{d$\_$type=SOL}
\State \textbf{if} $f(x^k+d^k)\le f(x^k)$ and $\|\nabla f(x^k+d^k)\|\le\epsilon_g$ \textbf{then} set $\alpha_k=1$;
\State \textbf{else} Find $\alpha_k=\theta^{j_k}$, where $j_k$ is the smallest nonnegative integer $j$ such that
\begin{equation}\label{ls-sol-stepsize}
f(x^k + \theta^j d^k) \le f(x^k) - \eta(\gamma_\nu(\epsilon_g)\epsilon_g)^{1/2}\theta^{2j}\|d^k\|^2;
\end{equation}
\State \textbf{end if}
\ElsIf{d$\_$type=NC}
\State Find $\alpha_k=\theta^{j_k}$, where $j_k$ is the smallest nonnegative integer $j$ such that
\begin{equation}\label{ls-nc-stepsize}
f(x^k+\theta^j d^k) \le f(x^k) - \eta\min\{1,\gamma_\nu(\epsilon_g)\}\theta^{2j}\|d^k\|^3/4;
\end{equation}
\ElsIf{d$\_$type=MEO}
\State Find $\alpha_k=\theta^{j_k}$, where $j_k$ is the smallest nonnegative integer $j$ such that
\begin{equation}\label{ls-meo-pd}
f(x^k+\theta^j d^k) \le f(x^k) - \eta\theta^{2j}\|d^k\|^3/2;
\end{equation}
\EndIf
\\\hrulefill
\\ \Comment{{\it  This part updates the next iterate.}}
\State Set $x^{k+1} = x^k + \alpha_k d^k$;
\EndFor
\end{algorithmic}
}
\end{algorithm}

The following theorem states the iteration and operation complexity of Algorithm~\ref{alg:NCG-pd} for finding an $\epsilon_g$-FOSP, whose proof is relegated to Section~\ref{subsec:proof1}.

\begin{theorem}\label{thm:c-pd}
Suppose that Assumption~\ref{asp:NCG-cmplxity} holds with some $H_\nu>0$ and $\nu\in[0,1]$, and $\epsilon_H$ is not provided for Algorithm~\ref{alg:NCG-pd}. Let $\epsilon_g \in (0,1)$ be given, $\fl$ and $U_H$ be given in \eqref{lwbd-Hgupbd}, $\gamma_\nu(\epsilon_g)$ be given  \eqref{gma-eps}, $\zeta$,  $\eta$, and $\theta$ be given in Algorithm~\ref{alg:NCG-pd}, and 
\begin{align}
& c_{\mathrm{sol}} := \eta \min\bigg\{\bigg(\frac{2}{4+\zeta + \sqrt{(4+\zeta)^2 + 1}}\bigg)^2, \frac{1}{6} \bigg(\frac{2(1-\eta)\theta}{3}\bigg)^2\bigg\},\qquad c_{\mathrm{nc}}:= \frac{\eta\theta^2}{4},\label{csol-cnc} \\
& K_1 := \left\lceil \frac{f(x^0)-\fl}{\min\{c_{\mathrm{sol}}, c_{\mathrm{nc}}\}}\gamma_\nu(\epsilon_g)^{1/2}\epsilon_g^{-3/2}\right\rceil + 1. \label{K1}
\end{align}
Then the following statements hold.
\begin{enumerate}[{\rm (i)}]
\item {\bf (iteration complexity)} Algorithm~\ref{alg:NCG-pd} terminates in at most $K_1$ iterations with 
\begin{align}\label{K1-order}
K_1 = \cO\big(H_\nu^{1/(1+\nu)}\epsilon_g^{-(2+\nu)/(1+\nu)}\big).
\end{align}
Moreover, its output $x^k$ satisfies $\|\nabla f(x^k)\|\le\epsilon_g$ for some $0\le k\le K_1$.
\item {\bf (operation complexity)} The main operations of Algorithm~\ref{alg:NCG-pd} consist of
\[
\widetilde{\cO}\big(H_\nu^{1/(1+\nu)}\epsilon_g^{-(2+\nu)/(1+\nu)}\min\big\{n,U_H^{1/2}/(H_\nu\epsilon_g^\nu)^{1/(2+2\nu)}\big\}\big)
\]
gradient evaluations and Hessian-vector products of $f$.
\end{enumerate}
\end{theorem}

\begin{remark}
\vspace{10mm}
\begin{enumerate}[{\rm (i)}]
\item The iteration complexity presented in Theorem~\ref{thm:c-pd}(i) matches the lower iteration complexity bound stated in \eqref{opt-2nd-ic} (see also \cite{cartis2011optimal,CaGoTo18}) for finding an $\epsilon_g$-FOSP of \eqref{ucpb} using a second-order method. Moreover, the operation complexity stated in Theorem~\ref{thm:c-pd}(ii) is a novel contribution to the literature. While some operation complexity results have been established in \cite[Corollaries 4 and 5]{zhang2023riemannian} for adaptive regularized Newton methods, those results only guarantee $x$ satisfying $\|\nabla f(x)\|\le\epsilon_g$ {\it with high probability}. In contrast,  the operation complexity in Theorem~\ref{thm:c-pd}(ii) is achieved by Algorithm~\ref{alg:NCG-pd} for deterministically finding an $\epsilon_g$-FOSP.

\item When $\nu=1$, the iteration and operation complexity of Algorithm~\ref{alg:NCG-pd} for finding an $\epsilon_g$-FOSP of \eqref{ucpb} are given by
\[
\cO\big(L_H^{1/2}\epsilon_g^{-3/2}\big)\quad \text{and}\quad \widetilde{\cO}\big(L_H^{1/2}\epsilon_g^{-3/2}\min\big\{n,U_H^{1/2}/(L_H\epsilon_g)^{1/4}\big\}\big),
\]
respectively, where $L_H$ is the Lipschitz constant of $\nabla^2 f$. These results demonstrate improved dependence on $L_H$ compared to the iteration and operation complexity achieved by the Newton-CG methods in \cite{hlp2023ncgal,RNW18} for finding an $\epsilon_g$-FOSP of \eqref{ucpb}, which are $\cO\big(L_H^{2}\epsilon_g^{-3/2}\big)$ and $\widetilde{\cO}\big(L_H^{2}\epsilon_g^{-3/2}\min\big\{n,U_H^{1/2}/\epsilon_g^{1/4}\big\}\big)$, respectively.
\end{enumerate}
\end{remark}

The next theorem establishes iteration and operation complexity of Algorithm~\ref{alg:NCG-pd} for finding a stochastic $(\epsilon_g,\epsilon_H)$-SOSP. Its proof is deferred to Section~\ref{subsec:proof1}.

\begin{theorem}\label{thm:c-pd-sosp}
Suppose that Assumption~\ref{asp:NCG-cmplxity} holds with some $H_\nu>0$ and $\nu\in(0,1]$, and $\epsilon_H\in(0,1)$ is provided for Algorithm~\ref{alg:NCG-pd}. 
Let $\epsilon_g \in (0,1)$ be given, $\fl$ and $U_H$ be given in \eqref{lwbd-Hgupbd}, $K_1$ be defined in \eqref{K1}, $\eta$ and $\theta$ be given in Algorithm~\ref{alg:NCG-pd},  and
\begin{align}
&c_{\mathrm{meo}}:=(\eta/2)\min\big\{1,\theta((1-\eta)/H_\nu)^{1/\nu}\big\}^2(1/2)^{(2+\nu)/\nu}, \label{cmeo} \\
& K_2 := \left\lceil \frac{f(x^0)-\fl}{c_{\mathrm{meo}}} \epsilon_H^{-(2+\nu)/\nu}\right\rceil + 1. \label{K2}
\end{align}
Then the following statements hold.
\begin{enumerate}[{\rm (i)}]
\item {\bf (iteration complexity)} Algorithm~\ref{alg:NCG-pd} terminates in at most $K_1 + 2K_2-1$ iterations with
\begin{align}\label{K1K2-order}
K_1 + 2K_2 -1 = \cO\big(H_\nu^{1/(1+\nu)}\epsilon_g^{-(2+\nu)/(1+\nu)}+H_\nu^{2/\nu}\epsilon_H^{-(2+\nu)/\nu}\big).
\end{align}
Moreover, its output $x^k$ satisfies $\|\nabla f(x^k)\|\le \epsilon_g$ deterministically and $\lambda_{\min}(\nabla^2 f(x^k)) \ge -\epsilon_H$ with probability at least $1-\delta$ for some $0\le k\le K_1 + 2 K_2 - 1$.
\item {\bf (operation complexity)} Algorithm~\ref{alg:NCG-pd} requires at most 
\begin{align*}
&\widetilde{\cO}\Big(\big(H_\nu^{1/(1+\nu)}\epsilon_g^{-(2+\nu)/(1+\nu)}+H_\nu^{2/\nu}\epsilon_H^{-(2+\nu)/\nu}\big)\min\big\{n,U_H^{1/2}/(H_\nu\epsilon_g^\nu)^{1/(2+2\nu)}\big\} \\
&\quad\,+ H_\nu^{2/\nu}\epsilon_H^{-(2+\nu)/\nu}\min\big\{n,(U_H/\epsilon_H)^{1/2}\big\}\Big)
\end{align*}
gradient evaluations and Hessian-vector products of $f$.
\end{enumerate}
\end{theorem}  

\begin{remark}
\begin{enumerate}[{\rm (i)}]
\item The operation complexity stated in Theorem~\ref{thm:c-pd-sosp}(ii) is a novel contribution to the literature. While similar operation complexity results have been established in \cite{zhang2023riemannian} for adaptive regularized Newton methods, those results only guarantee finding a point $x$ satisfying $\|\nabla f(x)\|\le\epsilon_g$ and $\lambda_{\min}(\nabla^2 f(x))\ge-\epsilon_H$, both {\it with high probability}. In contrast,  the operation complexity in Theorem~\ref{thm:c-pd-sosp}(ii) is achieved by Algorithm~\ref{alg:NCG-pd} for finding a point $x$ satisfying $\|\nabla f(x)\|\le\epsilon_g$ {\it deterministically} and $\lambda_{\min}(\nabla^2 f(x))\ge-\epsilon_H$ with high probability.

\item When $\nu=1$, the iteration and operation complexity results of Algorithm~\ref{alg:NCG} for finding a stochastic $(\epsilon_g,\epsilon_H)$-SOSP of \eqref{ucpb} are given by 
$\cO\big(L_H^{1/2}\epsilon_g^{-3/2}+L_H^2\epsilon_H^{-3}\big)$ and
\begin{align*}
\widetilde{\cO}\Big(\big(L_H^{1/2}\epsilon_g^{-3/2}+L_H^2\epsilon_H^{-3}\big)\min\big\{n,U_H^{1/2}/(L_H\epsilon_g)^{1/4}\big\} + L_H^2\epsilon_H^{-3}\min\big\{n,(U_H/\epsilon_H)^{1/2}\big\}\Big),    
\end{align*}
respectively, where $L_H$ is the Lipschitz constant  $\nabla^2 f$. When $\epsilon_H \ge (L_H\epsilon_g)^{1/2}$, these iteration and operation complexity results reduce to $\cO\big(L_H^{1/2}\epsilon_g^{-3/2}\big)$ and $ \widetilde{\cO}\big(L_H^{1/2}\epsilon_g^{-3/2}\min\big\{n,U_H^{1/2}/(L_H\epsilon_g)^{1/4}\big\}\big)$, respectively. These bounds exhibit improved dependence on $L_H$ compared to those achieved by the Newton-CG methods in \cite{hlp2023ncgal,RNW18} for finding a stochastic $(\epsilon_g,\epsilon_H)$-SOSP of \eqref{ucpb}, which are $\cO\big(L_H^{2}\epsilon_g^{-3/2}\big)$ and  $\widetilde{\cO}\big(L_H^{2}\epsilon_g^{-3/2}\min\big\{n,U_H^{1/2}/(L_H\epsilon_g)^{1/4}\big\}\big)$, respectively. 


\end{enumerate}
\end{remark}

\section{A parameter-free Newton-CG method for problem~\eqref{ucpb}}\label{sec:ncg}

In Section \ref{sec:pd-ncg}, we proposed a Newton-CG method (Algorithm~\ref{alg:NCG-pd}) for solving problem \eqref{ucpb}, assuming that the parameters $\nu$ and $H_\nu$ associated with the H\"older continuity of $\nabla^2 f$ are explicitly known. This method achieves the best-known iteration complexity for finding an $\epsilon_g$-FOSP deterministically and an $(\epsilon_g,\epsilon_H)$-SOSP with high probability, and its fundamental operations rely only on gradient evaluations and Hessian-vector products of $f$. However, this method requires explicit knowledge of
$\nu$ and $H_\nu$ to compute the quantity $\gamma_\nu(\epsilon_g)$, making it inapplicable to problem \eqref{ucpb} when these parameters are unknown. In addition, even when $\nu$ and $H_\nu$ are known, they may be overly conservative since they must satisfy \eqref{F-Hess-Lip} globally. This conservativeness can result in an excessively large $\gamma_\nu(\epsilon_g)$, potentially leading to slower practical convergence for Algorithm~\ref{alg:NCG-pd}. To address these challenges, we propose a parameter-free Newton-CG method (Algorithm~\ref{alg:NCG}), which incorporates an innovative backtracking scheme for locally estimating $\gamma_\nu(\epsilon_g)$. This method achieves a similar order of iteration and operation complexity as Algorithm~\ref{alg:NCG-pd}, but without requiring any prior knowledge of $\nu$ and $H_\nu$.

We now briefly describe the parameter-free Newton-CG method (Algorithm~\ref{alg:NCG}) for solving \eqref{ucpb}. At each outer iteration $k$, we perform the following operations.
\begin{itemize}
\item[(i)] If $x^k$ satisfies $\|\nabla f(x^k)\|>\epsilon_g$, we invoke the capped CG method (Algorithm~\ref{alg:capped-CG}) to solve a damped Newton system
\begin{align}\label{damp-n-trial}
\big(\nabla^2 f(x^k) + 2(\sigma_t\epsilon_g)^{1/2}I\big)d = -\nabla f(x^k),    
\end{align}
where $\sigma_t$ is a trial value replacing $\gamma_\nu(\epsilon_g)$ in \eqref{damp-n}. We then evaluate whether the current trial $\sigma_t$ appropriately estimates $\gamma_\nu(\epsilon_g)$ by performing several checks on the output  $d^t_k$ of Algorithm~\ref{alg:capped-CG} as follows.
\begin{itemize}
\item 
If both $f(x^k + d^t_k)\le f(x^k)$ and $\|\nabla f(x^k + d^t_k)\|\le\epsilon_g$ hold, the trial $\sigma_t$ is deemed an appropriate estimate of $\gamma_\nu(\epsilon_g)$, and $d_k^t$ is accepted as a suitable descent direction for generating the next iterate $x^{k+1}$.
\item 
If $6\|d_k^t\|< (\epsilon_g/\sigma_t)^{1/2}$, the trial $\sigma_t$  is considered an inappropriate estimate of  $\gamma_\nu(\epsilon_g)$. In this case,  $\sigma_t$  is increased by a ratio $r$, and the process is repeated with the updated $\sigma_t$.
\item 
If $6\|d_k^t\|\geq (\epsilon_g/\sigma_t)^{1/2}$, a line search 
is performed to determine whether a suitable step size exists for  $d_k^t$ to achieve sufficient reduction in $f$. If a suitable step size is found, the next iterate  $x^{k+1}$ is generated using this step size and the direction $d^t_k$. If no such step size exists,  $\sigma_t$  is increased by a ratio $r$ and the process is repeated with the updated $\sigma_t$.
\end{itemize}
\item[(ii)]
If $x^k$ satisfies $\|\nabla f(x^k)\|\le\epsilon_g$, similar to Algorithm~\ref{alg:NCG-pd}, this algorithm offers two options.
\begin{itemize}
\item When the tolerance $\epsilon_H\in(0,1)$ for second-order stationarity is not provided, Algorithm~\ref{alg:NCG} terminates with $x^k$ as an $\epsilon_g$-FOSP. 
\item When $\epsilon_H\in(0,1)$ is provided, a minimum eigenvalue oracle (Algorithm~\ref{pro:meo}) is invoked to either obtain a sufficiently negative curvature direction and generate the next iterate $x^{k+1}$ via a line search, or certify that $x^k$ is an $(\epsilon_g,\epsilon_H)$-SOSP with high probability and terminate the algorithm.
\end{itemize}
\end{itemize}

\begin{algorithm}
{\footnotesize
\caption{A parameter-free Newton-CG method for problem~\eqref{ucpb}}
\label{alg:NCG}
\begin{algorithmic}
\State {\bf input}: tolerance $\epsilon_g\in(0,1)$, starting point $x^0$, CG-accuracy parameter $\zeta\in(0,1)$, trial regularization parameter $\gamma_{-1}>0$, backtracking ratios $r>1, \theta\in(0,1)$, line-search parameter $\eta\in(0,1)$, probability parameter $\delta\in(0,1)$; {\bf optional input:} tolerance $\epsilon_H\in(0,1)$.
\For{$k=0,1,2,\ldots$}
\\\hrulefill 
\State \Comment{{\it This part aims to improve first-order stationarity by calling Algorithm~\ref{alg:capped-CG}.}}
\If{$\|\nabla f(x^k)\|>\epsilon_g$}
\State Set $H_k = \nabla^2 f(x^k)$, $g^k = \nabla f(x^k)$, and $\sigma_0 = \max\{\gamma_{-1},\gamma_{k-1}/r\}$; 
\For{$t=0,1,2,\ldots$}
\State Set $\sigma_t=r^t\sigma_0$;
\State Call Algorithm~\ref{alg:capped-CG} (Appendix~\ref{appendix:capped-CG}) with $H=H_k$, $\varepsilon=(\sigma_t\epsilon_g)^{1/2}$, $g=g^k$, accuracy parameter $\zeta$, and $U=0$ to obtain
\State outputs $d$, d$\_$type;
\If{d$\_$type=NC}
\State Set
\begin{equation*}
d_k^t= -\sgn(d^Tg^k)\max\{1,1/\sigma_t\}\frac{|d^TH_k d|}{\|d\|^3}d;
\end{equation*}
\Else\ \{d$\_$type=SOL\}
\State Set
\begin{equation*}
d_k^t= d;
\end{equation*}
\EndIf
\If{d\_type=SOL}
\State {\bf if} {$f(x^k + d^t_k)\le f(x^k)$ and $\|\nabla f(x^k + d^t_k)\|\le\epsilon_g$} {\bf then} set $j_t=0$ and {\bf break} the inner loop;
\State {\bf else if} $6\|d^t_k\|\ge (\epsilon_g/\sigma_t)^{1/2}$ {\bf then}
\State \  \quad Check whether there exists any nonnegative integer $j$ satisfying
\begin{align}
&\theta^j\ge \min\{1,2(1-\eta)\theta(\epsilon_g/\sigma_t)^{1/4}/(3\|d^t_k\|^{1/2})\},\label{sol-lwbd-thetaj}\\
&f(x^k + \theta^{j} d^t_k) \le f(x^k) - \eta(\sigma_t\epsilon_g)^{1/2}\theta^{2j}\|d_k^t\|^2;\label{ls-sol-stepsize-2}
\end{align}
\State\  \quad If such $j$ exists, set $j_t$ as the smallest nonnegative integer such that \eqref{ls-sol-stepsize-2} holds and {\bf break} the inner loop;
\State {\bf end if}
\ElsIf{d\_type=NC}
\State Check whether there exists any nonnegative integer $j$ satisfying
\begin{align}
&\theta^{j-1}\ge \min\{1,1/\sigma_t\},\label{nc-lwbd-thetaj}\\
&f(x^k + \theta^j d^t_k) \le f(x^k) - \eta\min\{1,\sigma_t\}\theta^{2j}\|d^t_k\|^3/4;\label{ls-nc-stepsize-2}
\end{align}
\State If such $j$ exists, set $j_t$ as the smallest nonnegative integer such that \eqref{ls-nc-stepsize-2} holds and {\bf break} the inner loop;
\EndIf
\EndFor
\State Set $(\alpha_k,\gamma_k,d^k)=(\theta^{j_t},\sigma_t,d^t_k)$;
\ElsIf{$\|\nabla f(x^k)\|\le\epsilon_g$ and $\epsilon_H$ is not provided}
\State Output $x^k$ and terminate;

\\\hrulefill 
\State \Comment{{\it This part aims to improve second-order stationarity by calling Algorithm~\ref{pro:meo}.}}
\ElsIf{$\|\nabla f(x^k)\|\le\epsilon_g$ and $\epsilon_H$ is provided}
\State Call Algorithm~\ref{pro:meo} (Appendix~\ref{appendix:meo}) with $H=\nabla^2 f(x^k)$, $\varepsilon=\epsilon_H$, and probability parameter $\delta$;
\If{Algorithm~\ref{pro:meo} certifies that $\lambda_{\min}(\nabla^2 f(x^k))\ge-\epsilon_H$}
\State Output $x^k$ and terminate;
\Else\ \{Sufficiently negative curvature direction $v$ returned by Algorithm~\ref{pro:meo}\}
\State Set $\gamma_k=\gamma_{k-1}$ and
\begin{equation*}
d^k \leftarrow -\mathrm{sgn}(v^T\nabla f(x^k))|v^T\nabla^2 f(x^k) v| v; 
\end{equation*}
\EndIf
\State Find $\alpha_k=\theta^{j_k}$, where $j_k$ is the smallest nonnegative integer $j$ such that
\begin{align*}
f(x^k+\theta^j d^k) \le f(x^k) - \eta \theta^{2j}\|d^k\|^3/2;    
\end{align*}
\EndIf
\\\hrulefill 
\State \Comment{{\it This part updates the next iterate.}}
\State Set $x^{k+1}=x^k+\alpha_k d^k$; 
\EndFor
\end{algorithmic}
}
\end{algorithm}

In what follows, we present the complexity results for Algorithm~\ref{alg:NCG}. For ease of reference, we define an {\it outer iteration} of Algorithm~\ref{alg:NCG} as one iteration that updates $x^k$ to $x^{k+1}$, and an {\it inner iteration} as one call to either Algorithm~\ref{alg:capped-CG} or Algorithm~\ref{pro:meo}. The following theorem establishes that the number of calls to Algorithm~\ref{alg:capped-CG} at each outer iteration of Algorithm~\ref{alg:NCG} is finite, ensuring that Algorithm~\ref{alg:NCG} is well-defined. The proof of this result is provided in Section~\ref{subsec:proof2}.


\begin{theorem}[{{\bf well-definedness of Algorithm~\ref{alg:NCG}}}]\label{thm:wd-ncg}
Suppose that Assumption~\ref{asp:NCG-cmplxity} holds. Let $\{\gamma_k\}$ be generated by Algorithm~\ref{alg:NCG}, $\gamma_\nu(\epsilon_g)$ be defined in \eqref{gma-eps}, and 
\begin{equation}\label{inner-bd}
\sigma(\epsilon_g):= \max\{\gamma_{-1},r\gamma_\nu(\epsilon_g)\},\qquad T :=\big\lceil\log(\sigma(\epsilon_g)/\gamma_{-1})/\log r\big\rceil_++2,
\end{equation}
where $\gamma_{-1}$ and $r$ are the inputs of Algorithm~\ref{alg:NCG}. Then, the number of calls of Algorithm~\ref{alg:capped-CG} at the $k$th iteration of Algorithm~\ref{alg:NCG} is at most $T$, and $\gamma_k \le \sigma(\epsilon_g)$ holds for all $k\ge0$. 
Moreover, the total number of calls of Algorithms~\ref{alg:capped-CG} and \ref{pro:meo} during the first $s$ outer iterations of Algorithm~\ref{alg:NCG} is at most $T+2s$.
\end{theorem}

The next theorem states the iteration and operation complexity of Algorithm~\ref{alg:NCG} for finding an $\epsilon_g$-FOSP. Its proof is deferred to Section~\ref{subsec:proof2}.

\begin{theorem}\label{thm:c-ncg}
Suppose that Assumption~\ref{asp:NCG-cmplxity} holds with some $H_\nu>0$ and $\nu\in[0,1]$, and $\epsilon_H$ is not provided for Algorithm~\ref{alg:NCG}. Let $\epsilon_g \in (0,1)$ be given, $\fl$, $U_H$, $c_{\mathrm{nc}}$, $\sigma(\epsilon_g)$, and $T$ be  respectively given in  \eqref{lwbd-Hgupbd}, \eqref{csol-cnc}, and \eqref{inner-bd}, $\eta$ and $\theta$ be given in Algorithm~\ref{alg:NCG-pd},  and
\begin{align}
& \hat{c}_{\mathrm{sol}}:= \frac{\eta}{6} \min\bigg\{\frac{1}{6}, \left(\frac{2(1-\eta)\theta}{3}\right)^2\bigg\}, \label{hat-csol-cnc} \\
&  \overline{K}_1 =\left\lceil \frac{f(x^0)-\fl}{\min\{\hat{c}_{\mathrm{sol}},{c}_{\mathrm{nc}}\}}\sigma(\epsilon_g)^{1/2}\epsilon_g^{-3/2}\right\rceil+1. \label{hat-K1} 
\end{align}
Then the following statements hold.
\begin{enumerate}[{\rm (i)}]
\item {\bf{(iteration complexity)}} Algorithm~\ref{alg:NCG} requires at most $\overline{K}_1$ outer iterations and $T + 2\overline{K}_1$ inner iterations, where 
\begin{align}
&\overline{K}_1 = \cO\big(H_\nu^{1/(1+\nu)} \epsilon_g^{-(2+\nu)/(1+\nu)}\big),\label{order-pf-KP1}\\
&T + 2\overline{K}_1 = {\cO}\big(H_\nu^{1/(1+\nu)} \epsilon_g^{-(2+\nu)/(1+\nu)}\big).\label{order-pf-TKP1}
\end{align}
Moreover, its output $x^k$ satisfies $\|\nabla f(x^k)\|\le\epsilon_g$ for some $0\le k \le \overline{K}_1$.
\item {\bf{(operation complexity)}} Algorithm~\ref{alg:NCG} requires at most
\[
\widetilde{\cO}\Big(\min\big\{n, U_H^{1/2}/(H_\nu\epsilon_g^\nu)^{1/(2+2\nu)}\big\}H_\nu^{1/(1+\nu)}\epsilon_g^{-(2+\nu)/(1+\nu)}\Big)
\]
gradient evaluations and Hessian-vector products of $f$.
\end{enumerate}
\end{theorem}

\begin{remark}
From Theorems~\ref{thm:c-pd} and \ref{thm:c-ncg}, we observe that Algorithm~\ref{alg:NCG} achieves the same order of iteration and operation complexity as Algorithm~\ref{alg:NCG-pd} for finding an $\epsilon_g$-FOSP of problem \eqref{ucpb}. Moreover, the iteration complexity matches the lower complexity bound stated in \eqref{opt-2nd-ic} (see also \cite{cartis2011optimal,CaGoTo18}).
\end{remark}

The following theorem presents iteration and operation complexity of Algorithm~\ref{alg:NCG} for finding a stochastic $(\epsilon_g,\epsilon_H)$-SOSP. Its proof is deferred to Section~\ref{subsec:proof2}.

\begin{theorem}\label{thm:c-ncg-sosp}
Suppose that Assumption~\ref{asp:NCG-cmplxity} holds with some $H_\nu>0$ and $\nu\in(0,1]$, and $\epsilon_H\in(0,1)$ is provided for Algorithm~\ref{alg:NCG}. Let $U_H$, $K_2$, $T$ and $\overline{K}_1$ be defined in  \eqref{lwbd-Hgupbd},  \eqref{K2},  \eqref{inner-bd} and \eqref{hat-K1}, respectively. Then the following statements hold.
\begin{enumerate}[{\rm (i)}]
\item {\bf{(iteration complexity)}} Algorithm~\ref{alg:NCG} requires at most $\overline{K}_1 + 2K_2-1$ outer iterations and $T + 2\overline{K}_1 + 4K_2- 2$ inner iterations, where 
\begin{align}
&\overline{K}_1 + 2K_2-1=\cO\big(H_\nu^{1/(1+\nu)}\epsilon_g^{-(2+\nu)/(1+\nu)}+H_\nu^{2/\nu}\epsilon_H^{-(2+\nu)/\nu}\big),\label{outer-iter-ncg-pf}\\
&T + 2\overline{K}_1 + 4K_2- 2= {\cO}\big(H_\nu^{1/(1+\nu)} \epsilon_g^{-(2+\nu)/(1+\nu)}+H_\nu^{2/\nu}\epsilon_H^{-(2+\nu)/\nu}\big).\label{inner-iter-ncg-pf}
\end{align}
Also, its output $x^k$ satisfies $\|\nabla f(x^k)\|\le\epsilon_g$ deterministically and $\lambda_{\min}(\nabla^2 f(x^k))\ge-\epsilon_H$ with probability at least $1-\delta$ for some $0\le k\le \overline{K}_1 + 2K_2 - 1$.
\item {\bf{(operation complexity)}} Algorithm~\ref{alg:NCG} requires at most
\[
\widetilde{\cO}\Big(\min\big\{n,U_H^{1/2}/\epsilon_g^{1/4}\big\} \big(H_\nu^{1/(1+\nu)} \epsilon_g^{-(2+\nu)/(1+\nu)} +H_\nu^{2/\nu}\epsilon_H^{-(2+\nu)/\nu}\big) + \min\big\{n,(U_H/\epsilon_H)^{1/2}\big\}H_\nu^{2/\nu}\epsilon_H^{-(2+\nu)/\nu}\Big)
\]
gradient evaluations and Hessian-vector products of $f$.
\end{enumerate}
\end{theorem}

\begin{remark}
From Theorems~\ref{thm:c-pd-sosp} and \ref{thm:c-ncg-sosp}, we observe that Algorithm~\ref{alg:NCG} achieves the same order of iteration complexity as Algorithm~\ref{alg:NCG-pd} for finding an $(\epsilon_g,\epsilon_H)$-SOSP of problem \eqref{ucpb} with high probability.     
\end{remark}

\section{Numerical results}\label{sec:num}
In this section, we conduct preliminary numerical experiments to test the performance of our parameter-free Newton-CG method (Algorithm~\ref{alg:NCG}), and compare it with  the adaptive cubic regularized Newton method (Universal Method II) in \cite{GeNe17NH}. All the algorithms are coded in Matlab,  and all the computations are performed on a laptop with a 2.20 GHz Intel Core i9-14900HX processor and 32 GB of RAM.

\subsection{Infeasibility detection problem}\label{subsec:ifeas}
In this subsection, we consider the infeasibility detection problem (see \cite{ByCuNo10}):
\begin{equation}\label{feas-dect}
\min_{x\in\bR^n} \frac{1}{m}\sum_{i=1}^m \left(x^TA_ix + b_i^Tx+ c_i\right)_+^p,
\end{equation}
where $p> 2$, $A_i\in\bR^{n\times n}$, $b_i\in\bR^n$, and $c_i\in\bR$ for $1\le i\le m$.

Our goal is to find a $10^{-4}$-FOSP of problem~\eqref{feas-dect} for the generated instances using Algorithm~\ref{alg:NCG} and the adaptive cubic regularized Newton method \cite[Universal Method II]{GeNe17NH}, and compare their performance. For the adaptive cubic regularized Newton method, we employ a gradient descent method to solve its cubic regularized subproblems, as suggested in \cite{CaDu19}. The initial point for the gradient descent method is uniformly selected from the unit sphere, and the tolerances for the subproblems decrease over iterations. For both methods, we initialize with $x^0=(0,\ldots,0)^T$, and choose the following parameter settings, which appear to suit each method well in terms of computational performance:
\begin{itemize}
\item $(\zeta,\gamma_{-1},\theta,r,\eta)=(0.5,10,0.5,2,0.01)$ for Algorithm~\ref{alg:NCG};
\item $H_0=10$ for the adaptive cubic regularized Newton method.
\end{itemize}

The computational results for Algorithm~\ref{alg:NCG} and the adaptive cubic regularized Newton method (abbreviated as A-CRN) applied to problem~\eqref{feas-dect} are presented in Table~\ref{table:feas-dect}. The first three columns of the table list the values of $n$, $m$, and $p$, respectively.  The remaining columns present the average final objective value, the average CPU time, and the average total number of subproblems over 10 random instances for each triple $(n,m,p)$. Here, a subproblem refers to either one cubic regularized subproblem solved by A-CRN or one damped Newton system solved by Algorithm~\ref{alg:NCG}. The results show that both methods produce solutions with comparable final objective values. However, Algorithm~\ref{alg:NCG} significantly outperforms A-CRN \cite{GeNe17NH} in terms of CPU time.

\begin{table}[t]
{\small
\centering
\begin{tabular}{ccc||ll||ll||ll}
\hline
& & &\multicolumn{2}{c||}{Objective value} &\multicolumn{2}{c||}{CPU time (seconds)} & \multicolumn{2}{c}{Total subproblems} \\
$n$ & $m$ & $p$ & Algorithm~\ref{alg:NCG} & A-CRN & Algorithm~\ref{alg:NCG} & A-CRN & Algorithm~\ref{alg:NCG} & A-CRN \\ \hline
100 & 10 & 2.25 & 7.1$\times10^{-15}$ & 1.7$\times10^{-14}$ & 0.01 & 0.23 & 10.3 & 28.5 \\ 
100 & 10 & 2.5  & 1.2$\times10^{-13}$ & 1.8$\times10^{-13}$ & 0.02 & 0.28 & 11.1 & 35.8 \\ 
100 & 10 & 2.75 & 7.2$\times10^{-13}$ & 3.0$\times10^{-12}$ & 0.02 & 0.27 & 13.4 & 41.7 \\ 
100 & 10 & 3    & 1.5$\times10^{-12}$ & 4.5$\times10^{-12}$ & 0.04 & 0.37 & 13.1 & 51.8 \\ 
500 & 50 & 2.25 & 1.8$\times10^{-16}$ & 2.6$\times10^{-15}$ & 3.15 & 7.51 & 11.5 & 45.6  \\ 
500 & 50 & 2.5  & 4.8$\times10^{-15}$ & 5.9$\times10^{-15}$ & 6.48 & 19.14 & 13.2 & 53.0 \\ 
500 & 50 & 2.75 & 7.0$\times10^{-14}$ & 3.7$\times10^{-14}$ & 5.29 & 16.73 & 14.2 & 58.8 \\ 
500 & 50 & 3.0  & 2.3$\times10^{-13}$ & 3.3$\times10^{-13}$ & 3.61 & 8.92 & 15.3 & 63.8 \\
1000 & 100 & 2.25 & 1.9$\times10^{-18}$ & 3.3$\times10^{-17}$ & 12.82 & 33.82 & 11.2 & 49.6 \\ 
1000 & 100 & 2.5  & 3.1$\times10^{-15}$ & 6.3$\times10^{-15}$ & 16.23 & 37.34 & 14.4 & 58.9 \\ 
1000 & 100 & 2.75 & 6.8$\times10^{-15}$ & 3.0$\times10^{-15}$ & 17.75 & 39.02 & 15.3 & 63.5 \\ 
1000 & 100 & 3    & 2.8$\times10^{-14}$ & 1.8$\times10^{-14}$ & 18.67 & 43.51 & 16.5 & 67.2 \\\hline 
\end{tabular}
\caption{Numerical results for problem~\eqref{feas-dect}}
\label{table:feas-dect}
}
\end{table}

\subsection{Single-layer neural networks problem}
In this subsection, we consider the problem of training single-layer rectified power unit (RePU)  neural networks (see \cite{LiTaYu19RePU}):
\begin{equation}\label{nn}
\min_{x\in\bR^n}{\frac{1}{m}}\sum_{i=1}^m \phi((a_i^Tx)_+^p - b_i),
\end{equation}
where $p> 2$, $\phi(t)=t^2/(1+t^2)$ is a nonconvex loss function (see \cite{BeTu74,CaDuHiSi17}), $a_i\in\bR^n$, and $b_i\in\bR$ for $1\le i\le m$.

For each triple $(n,m,p)$, we randomly generate 10 instances of problem~\eqref{nn}. In particular, we first randomly generate $a_i$, $1\le i\le m$, with all its components sampled from the standard normal distribution. We then randomly generate $\bar{b}_i$, $1\le i\le m$, according to the standard normal distribution, and set $b_i=|\bar{b}_i|$ for $1\le i\le m$.

\begin{table}[t]
{\small
\centering
\begin{tabular}{ccc||ll||ll||ll}
\hline
& & &\multicolumn{2}{c||}{Objective value} &\multicolumn{2}{c||}{CPU time (seconds)} & \multicolumn{2}{c}{Total subproblems} \\
$n$ & $m$ & $p$ & Algorithm~\ref{alg:NCG} & A-CRN & Algorithm~\ref{alg:NCG} & A-CRN & Algorithm~\ref{alg:NCG} & A-CRN \\ \hline
100 & 20 & 2.25 & 0.09 & 0.10 & 0.09 & 0.46 & 56.9 & 218.4 \\ 
100 & 20 & 2.5  & 0.09 & 0.10 & 0.10 & 0.52 & 60.1 & 242.8 \\ 
100 & 20 & 2.75 & 0.09 & 0.10 & 0.08 & 0.87 & 64.6 & 382.2 \\ 
100 & 20 & 3    & 0.10 & 0.12 & 0.12 & 0.61 & 61.5 & 325.9 \\ 
500 & 100 & 2.25 & 0.09 & 0.10 & 8.25 & 10.96 & 148.6 & 422.5 \\ 
500 & 100 & 2.5  & 0.10 & 0.11 & 9.16 & 14.85 & 153.3 & 507.7 \\ 
500 & 100 & 2.75 & 0.10 & 0.11 & 9.67 & 16.91 & 148.9 & 567.1 \\ 
500 & 100 & 3    & 0.11 & 0.12 & 11.04 & 18.83 & 164.1 & 646.7 \\ 
1000 & 200 & 2.25 & 0.10 & 0.11 & 29.53 & 34.92 & 228.7 & 631.5  \\ 
1000 & 200 & 2.5  & 0.10 & 0.11 & 34.72 & 48.42 & 241.4 & 800.2 \\ 
1000 & 200 & 2.75 & 0.11 & 0.12 & 39.81 & 70.43 & 249.1 & 856.8 \\ 
1000 & 200 & 3.0  & 0.12 & 0.13 & 46.46 & 106.92 & 279.2 & 1208.9 \\\hline 
\end{tabular}
\caption{Numerical results for problem~\eqref{nn}}
\label{table:nn}
}
\end{table}

Our goal is to find a $10^{-4}$-FOSP of problem~\eqref{nn} fof problem~\eqref{nn} for the generated instances using Algorithm~\ref{alg:NCG} and the adaptive cubic regularized Newton method \cite[Universal Method II]{GeNe17NH}, and compare their performance.  For the adaptive cubic regularized Newton method, we employ a gradient descent approach to solve its cubic regularized subproblems, as suggested in \cite{CaDu19}. The initial point for the gradient descent method is uniformly selected from the unit sphere, with tolerances for the subproblems decreasing over iterations. For both methods, we initialize $x_0=(1/n,\ldots,1/n)^T$, and adopt the same parameter settings for Algorithm~\ref{alg:NCG} and the cubic regularized Newton method as those specified in Subsection~\ref{subsec:ifeas}.

The computational results for Algorithm~\ref{alg:NCG} and the adaptive cubic regularized Newton method (abbreviated as A-CRN) for solving problem~\eqref{nn} are presented in Table~\ref{table:nn}. The first three columns of the table list the values of $n$, $m$, and $p$, respectively. The remaining columns present the average final objective, average CPU time, and average total number of subproblems over 10 random instances for each triple $(n,m,p)$. Here, a subproblem refers to either a cubic regularized subproblem solved by A-CRN or a damped Newton system solved by Algorithm~\ref{alg:NCG}. We observe that both methods find a $10^{-4}$-FOSP of \eqref{nn} with comparable objective values. However,  Algorithm~\ref{alg:NCG} is substantially faster than A-CRN \cite{GeNe17NH}. 

\section{Proof of the main results}\label{sec:proof}
In this section we provide a proof of our main results presented in Sections~\ref{sec:pd-ncg} and \ref{sec:ncg}, which are particularly Theorems~\ref{thm:c-pd}-\ref{thm:c-ncg-sosp}.

To proceed, we first establish several technical lemmas. The following lemma demonstrates that $\nabla f$ admits a first-order approximation with a controllable error, which will play a crucial role in our subsequent analysis. This result is inspired by \cite{Ne15}, where it was shown that a function with a H\"older continuous gradient admits a first-order approximation with a controllable error.

\begin{lemma}\label{lem:inexact}
Under Assumption~\ref{asp:NCG-cmplxity}, the following inequality holds for any $\delta>0$: 
\begin{align}
\|\nabla f(y)-\nabla f(x)-\nabla^2 f(x)(y-x)\|\le \frac{1}{2}L(\delta)\|y-x\|^2 + \delta,\quad \forall x,y\in\Omega,\label{i-HL-1}    
\end{align}
where
\begin{equation}\label{L1-L2}
L(\delta) = \left(\frac{1-\nu}{2\delta(1+\nu)}\right)^{\frac{1-\nu}{1+\nu}}H_\nu^{\frac{2}{1+\nu}}, \quad\forall\delta>0.\footnote{By convention, $0^0$ is set to $1$ throughout this paper.}
\end{equation}
\end{lemma}

\begin{proof}
When $\nu=1$, it follows from \eqref{apx-nxt-grad} that \eqref{i-HL-1} holds. For the rest of the proof, we suppose that $\nu\in(0,1)$. By Young's inequality, one has that  $\tau s \le \tau^p/p + s^q/q$ for all $\tau,s\ge0$, where $p,q\ge1$ satisfy $1/p+1/q=1$. Taking $\tau=t^{1+\nu}$, $p=2/(1+\nu)$, and $q=2/(1-\nu)$, we obtain that 
\[
t^{1+\nu} \le \frac{1+\nu}{2s} t^2 + \frac{1-\nu}{2} s^{\frac{1+\nu}{1-\nu}},\quad \forall t\ge0,s>0.
\]
Further, letting $t=\|y-x\|$, $s=\left(\frac{2\delta(1+\nu)}{H_\nu(1-\nu)}\right)^{\frac{1-\nu}{1+\nu}}$, and multiplying both sides of the above inequality by $H_\nu/(1+\nu)$, we have
\[
\frac{H_\nu}{1+\nu}\|y-x\|^{1+\nu} \le \frac{1}{2}\left(\frac{1-\nu}{2\delta(1+\nu)}\right)^{\frac{1-\nu}{1+\nu}} H_\nu^{\frac{2}{1+\nu}}\|y-x\|^2 + \delta,
\]
which along with \eqref{apx-nxt-grad} and \eqref{L1-L2} implies that \eqref{i-HL-1} holds.
\end{proof}

The following lemma provides a useful property of $L(\cdot)$.

\begin{lemma}
For any $a\ge2$, we have
\begin{align}
L(\epsilon_g/a)\le a\gamma_\nu(\epsilon_g)/8,\label{L1-bd-gma}
\end{align}
where $\gamma_\nu(\cdot)$ and $L(\cdot)$ are defined in \eqref{gma-eps} and \eqref{L1-L2}, respectively.
\end{lemma}

\begin{proof}
By $a\ge2$, $\nu\in[0,1]$, \eqref{gma-eps}, and \eqref{L1-L2}, one has
\[
L\left(\frac{\epsilon_g}{a}\right)\overset{\eqref{L1-L2}}{=} \left(\frac{a(1-\nu)}{2\epsilon_g(1+\nu)}\right)^{\frac{1-\nu}{1+\nu}}H_\nu^{\frac{2}{1+\nu}}\epsilon_g^{-\frac{1-\nu}{1+\nu}}\le \frac{a}{2} H_\nu^{\frac{2}{1+\nu}}\epsilon_g^{-\frac{1-\nu}{1+\nu}} \overset{\eqref{gma-eps}}{=} \frac{a}{8}\gamma_{\nu}(\epsilon_g),
\]
where the first inequality is due to $\nu\in[0,1]$ and $\alpha^\alpha\le 1$ for all $\alpha\in[0,1]$, and the second inequality is due to $\nu\in[0,1]$ and $a\ge 2$. Hence, the conclusion of this lemma holds.
\end{proof}

The next lemma provides two equivalent reformulations of the inequality $\gamma\ge\gamma_\nu(\epsilon_g)$. It will be repeatedly used in our subsequent analysis.

\begin{lemma}\label{lem:tech-gammamu-eps}
Let $\gamma_\nu(\epsilon_g)$ be defined in \eqref{gma-eps}. Then, $\gamma\ge\gamma_\nu(\epsilon_g)$ is equivalent to each of the following two inequalities:
\begin{align}
&(\gamma\epsilon_g)^{1/2}/H_\nu \ge 2^{1+\nu}(\epsilon_g/\gamma)^{\nu/2},\label{rela:Hnu-gammanu-1}\\
&(\gamma\epsilon_g)^{(1-\nu)/2}/H_\nu \ge 2^{1+\nu}/\gamma^\nu.\label{rela:Hnu-gammanu-2}
\end{align}    
\end{lemma}
\begin{proof}
Dividing both sides of \eqref{rela:Hnu-gammanu-1} by $\epsilon_g^{1/2}/(H_\nu\gamma^{\nu/2})$ and dividing both sides of \eqref{rela:Hnu-gammanu-2} by $\epsilon_g^{(1-\nu)/2}/(H_\nu\gamma^\nu)$, we observe that  \eqref{rela:Hnu-gammanu-1} and \eqref{rela:Hnu-gammanu-2} are both equivalent to $\gamma^{(1+\nu)/2}\ge 2^{1+\nu}H_\nu\epsilon_g^{(\nu-1)/2}$. Taking the $\left(\frac{1+\nu}{2}\right)$th root of this inequality, it becomes
\begin{align*}
\gamma \ge 4H_\nu^{2/(1+\nu)}\epsilon_g^{-(1-\nu)/(1+\nu)}\overset{\eqref{gma-eps}}{=}\gamma_\nu(\epsilon_g).
\end{align*}
Hence, the conclusion of this lemma holds.
\end{proof}

\subsection{Proof of the main results in Section~\ref{sec:pd-ncg}}\label{subsec:proof1}

In this subsection, we first establish several technical lemmas and then use them to prove Theorems~\ref{thm:c-pd} and \ref{thm:c-pd-sosp}.


The following lemma presents some useful properties of the output of Algorithm~\ref{alg:capped-CG} when applied to solving the damped Newton system~\eqref{damp-n}. It is a direct consequence of Lemma~\ref{lem:ppt-cg} given in Appendix~\ref{appendix:capped-CG}.

\begin{lemma}\label{lem:SOL-NC-ppty}
Suppose that Assumption~\ref{asp:NCG-cmplxity} holds and the direction $d^k$ results from Algorithm~\ref{alg:capped-CG} with a type specified in d$\_$type at some iteration $k$ of Algorithm~\ref{alg:NCG-pd}. Then the following statements hold.
\begin{enumerate}[{\rm (i)}]
\item If d$\_$type=SOL, then  $d^k$ satisfies
\begin{eqnarray}
&(\gamma_\nu(\epsilon_g)\epsilon_g)^{1/2}\|d^k\|^2\le (d^k)^T(\nabla^2 f(x^k)+2(\gamma_\nu(\epsilon_g)\epsilon_g)^{1/2}I)d^k,\label{SOL-ppty-1}\\
&\|d^k\|\le 1.1(\gamma_\nu(\epsilon_g)\epsilon_g)^{-1/2}\|\nabla f(x^k)\|,\label{SOL-ppty-upbd}\\
&(d^k)^T\nabla f(x^k)=-(d^k)^T(\nabla^2f(x^k)+2 (\gamma_\nu(\epsilon_g)\epsilon_g)^{1/2}I)d^k,\label{SOL-ppty-2}\\
&\|(\nabla^2f(x^k)+2(\gamma_\nu(\epsilon_g)\epsilon_g)^{1/2}I)d^k+\nabla f(x^k)\|\le \zeta(\gamma_\nu(\epsilon_g)\epsilon_g)^{1/2}\|d^k\|/2.\label{SOL-ppty-3}    
\end{eqnarray}
\item If d$\_$type=NC, then $d^k$ satisfies $(d^k)^T\nabla f(x^k)\le0$ and
\begin{equation*}
(d^k)^T\nabla^2 f(x^k) d^k/\|d^k\|^2=-\min\{1,\gamma_\nu(\epsilon_g)\}\|d^k\|\le-(\gamma_\nu(\epsilon_g)\epsilon_g)^{1/2}.
\end{equation*}
\end{enumerate}
\end{lemma}

The next lemma shows that when the search direction $d^k$ in Algorithm~\ref{alg:NCG-pd} is of type `SOL', either both $f(x^k+d^k)\le f(x^k)$ and $\|\nabla f(x^k+d^k)\|\le\epsilon_g$ hold, or $\|d^k\|$ is uniformly bounded below.

\begin{lemma}\label{lem:next-FOSP-or-large-step}
Suppose that Assumption~\ref{asp:NCG-cmplxity} holds and the direction $d^k$ results from Algorithm~\ref{alg:capped-CG} with d$\_$type=SOL at some iteration $k$ of Algorithm~\ref{alg:NCG-pd}. Then, either both $f(x^k+d^k)\le f(x^k)$ and $\|\nabla f(x^k+d^k)\|\le\epsilon_g$ hold, or $6\|d^k\| \ge (\epsilon_g/\gamma_\nu(\epsilon_g))^{1/2}$ holds, where $\gamma_\nu(\epsilon_g)$ is defined in \eqref{gma-eps}.
\end{lemma}

\begin{proof}
Since $d^k$ results from Algorithm~\ref{alg:capped-CG} with d$\_$type=SOL, we see that $\|\nabla f(x^k)\|>\epsilon_g$ and \eqref{SOL-ppty-1}-\eqref{SOL-ppty-3} hold for $d^k$. Moreover, by $\|\nabla f(x^k)\|>\epsilon_g$ and \eqref{SOL-ppty-3}, we conclude that $d^k\neq 0$.
To prove this lemma, it suffices to show that both $f(x^k+d^k)\le f(x^k)$ and $\|\nabla f(x^k+d^k)\|\le\epsilon_g$ hold under the condition $6\|d^k\|< (\epsilon_g/\gamma_\nu(\epsilon_g))^{1/2}$. To this end,  we assume that $6\|d^k\|< (\epsilon_g/\gamma_\nu(\epsilon_g))^{1/2}$ holds throughout the remainder of this proof. 
    
We first prove $f(x^k+d^k)\le f(x^k)$. Suppose for contradiction that $f(x^k+d^k)> f(x^k)$. Let $\varphi(\alpha)=f(x^k + \alpha d^k)$ for all $\alpha$. Then $\varphi(1)>\varphi(0)$. Also, since $d^k\neq0$, one has
\[
\varphi^\prime(0) = \nabla f(x^k)^Td^k \overset{\eqref{SOL-ppty-2}}{=} - (d^k)^T(\nabla^2 f(x^k)+2(\gamma_\nu(\epsilon_g)\epsilon_g)^{1/2}I)d^k\overset{\eqref{SOL-ppty-1}}{\le} -(\gamma_\nu(\epsilon_g)\epsilon_g)^{1/2}\|d^k\|^2<0.
\]
In view of these and $\varphi(1)>\varphi(0)$, one can observe that there exists a local minimizer $\alpha_*\in(0,1)$ of $\varphi$ such that $\varphi^\prime(\alpha_*)=\nabla f(x^k+\alpha_*d^k)^Td^k=0$ and $\varphi(\alpha_*)<\varphi(0)$. Notice that $f$ is descent along the iterates of Algorithm~\ref{alg:NCG-pd}  and hence $f(x^k)\le f(x^0)$. This together with $\varphi(\alpha_*)<\varphi(0)$ implies that $f(x^k+\alpha_*d^k)<f(x^k)\le f(x^0)$. Hence, \eqref{apx-nxt-grad} holds for $x=x^k$ and $y=x^k+\alpha_*d^k$. Using this, $\alpha_*\in(0,1)$, \eqref{SOL-ppty-1}, \eqref{SOL-ppty-2}, and $\nabla f(x^k+\alpha_*d^k)^Td^k=0$, we deduce that
\begin{equation*}
\begin{array}{l}
\frac{\alpha_*^{1+\nu}H_\nu}{1+\nu}\|d^k\|^{2+\nu}\overset{\eqref{apx-nxt-grad}}{\ge}\|d^k\|\|\nabla f(x^{k}+\alpha_*d^k)-\nabla f(x^k)-\alpha_*\nabla^2 f(x^k)d^k\|\\[6pt]
\ge(d^k)^T(\nabla f(x^k+\alpha_*d^k)-\nabla f(x^k)-\alpha_*\nabla^2 f(x^k)d^k)=-(d^k)^T\nabla f(x^k)-\alpha_*(d^k)^T\nabla^2 f(x^k)d^k\\[1pt]
\overset{\eqref{SOL-ppty-2}}{=}(1-\alpha_*)(d^k)^T(\nabla^2 f(x^k)+2(\gamma_\nu(\epsilon_g)\epsilon_g)^{1/2}I)d^k +2 \alpha_*(\gamma_\nu(\epsilon_g)\epsilon_g)^{1/2}\|d^k\|^2\\
\overset{\eqref{SOL-ppty-1}}{\ge}(1+\alpha_*)(\gamma_\nu(\epsilon_g)\epsilon_g)^{1/2}\|d^k\|^2\ge (\gamma_\nu(\epsilon_g)\epsilon_g)^{1/2}\|d^k\|^2.
\end{array}
\end{equation*}
Recall that $\|d^k\|\neq 0$. Dividing both sides of the last inequality by $H_\nu\|d^k\|^2/(1+\nu)$, we obtain that 
\begin{align}\label{equi-alphastar-ge}
\|d^k\|^\nu \ge \alpha_*^{1+\nu} \|d^k\|^\nu \ge (1+\nu)(\gamma_\nu(\epsilon_g)\epsilon_g)^{1/2}/H_\nu,
\end{align}
where the first inequality is due to $\alpha_*\in(0,1)$. In addition, letting $\gamma=\gamma_\nu(\epsilon_g)$ in Lemma~\ref{lem:tech-gammamu-eps}, we obtain from \eqref{rela:Hnu-gammanu-1} that $(\gamma_\nu(\epsilon_g)\epsilon_g)^{1/2}/H_\nu \ge 2^{1+\nu}(\epsilon_g/\gamma_\nu(\epsilon_g))^{\nu/2}$, which together with \eqref{equi-alphastar-ge} implies that
\begin{equation*}
\|d^k\|^\nu \ge (1+\nu) 2^{1+\nu}(\epsilon_g/\gamma_\nu(\epsilon_g))^{\nu/2}.
\end{equation*}
Using this and $6\|d^k\|<(\epsilon_g/\gamma_\nu(\epsilon_g))^{1/2}$, we obtain that $(\epsilon_g/\gamma_\nu(\epsilon_g))^{\nu/2}/6^\nu > (1+\nu) 2^{1+\nu}(\epsilon_g/\gamma_\nu(\epsilon_g))^{\nu/2}$, which yields $(1+\nu) 2^{1+\nu}6^\nu<1$. This contradicts the fact $\nu\in [0,1]$. Hence, $f(x^k+d^k)\le f(x^k)$ holds.

We next prove $\|\nabla f(x^k + d^k)\|\le\epsilon_g$. As shown above, $f(x^k+d^k)\le f(x^k)$ holds, which together with $f(x^k) \le f(x^0)$ implies that $f(x^k+d^k)\le f(x^k)\le f(x^0)$.  It then follows that \eqref{i-HL-1} holds for $x=x^k$ and $y=x^k+d^k$. By this, \eqref{SOL-ppty-3}, and $6\|d^k\| < (\epsilon_g/\gamma_\nu(\epsilon_g))^{1/2}$, one has
\begin{align}
\|\nabla f(x^k + d^k)\| & \le \|\nabla f(x^k+d^k)-\nabla f(x^k)-\nabla^2 f(x^k)d^k\|
 +\|(\nabla^2 f(x^k)+2(\gamma_\nu(\epsilon_g)\epsilon_g)^{1/2}I)d^k +\nabla f(x^k)\| \nonumber\\
&\quad+2(\gamma_\nu(\epsilon_g)\epsilon_g)^{1/2}\|d^k\|\nonumber\\
& \le  \frac{L(\epsilon_g/2)}{2}\|d^k\|^2 + \frac{\epsilon_g}{2} + \frac{4+\zeta}{2}(\gamma_\nu(\epsilon_g)\epsilon_g)^{1/2}\|d^k\| < \frac{L(\epsilon_g/2)}{72}\frac{\epsilon_g}{\gamma_\nu(\epsilon_g)} + \frac{\epsilon_g}{2} + \frac{4 + \zeta}{12} \epsilon_g,\label{gradxkdk-sol-1} 
\end{align}
where the second inequality follows from \eqref{i-HL-1} and \eqref{SOL-ppty-3}, and the third inequality is due to $6\|d^k\| < (\epsilon_g/\gamma_\nu(\epsilon_g))^{1/2}$. Also, notice from \eqref{L1-bd-gma} with $a=2$ that $L(\epsilon_g/2)\le \gamma_\nu(\epsilon_g)/4$. Using this, \eqref{gradxkdk-sol-1}, and $\zeta\in(0,1)$, we obtain that
\[
\|\nabla f(x^k + d^k)\| \le \frac{\epsilon_g}{288} + \frac{\epsilon_g}{2} + \frac{4 + \zeta}{12} \epsilon_g  < \epsilon_g.
\]
Hence, $\|\nabla f(x^k + d^k)\|\le\epsilon_g$ holds as desired. 
\end{proof}

The next lemma shows that when the search direction $d^k$ in Algorithm~\ref{alg:NCG-pd} is of type `SOL', the line search step results in a sufficient reduction in $f$.

\begin{lemma}\label{lem:sol-1}
Suppose that Assumption~\ref{asp:NCG-cmplxity} holds and the direction $d^k$ results from  Algorithm~\ref{alg:capped-CG} with d$\_$type=SOL at some iteration $k$ of Algorithm~\ref{alg:NCG-pd}. Let $U_g$, $\gamma_\nu(\epsilon_g)$, and $c_{\mathrm{sol}}$ be defined in \eqref{lwbd-Hgupbd}, \eqref{gma-eps}, and \eqref{csol-cnc}, respectively. Then the following statements hold.
\begin{enumerate}[{\rm (i)}]
\item The step length $\alpha_k$ is well-defined, and moreover, $\alpha_k\ge \min\{1,(1-\eta)\theta(\epsilon_g/U_g)^{1/2}/2\}$.
\item The next iterate $x^{k+1} = x^k + \alpha_k d^k$ satisfies either $\|\nabla f(x^{k+1})\|\le\epsilon_g$ or
\begin{align}\label{ineq:descent-sol-pd}
f(x^k) - f(x^{k+1}) \ge {c}_{\mathrm{sol}} \epsilon_g^{3/2}/\gamma_\nu(\epsilon_g)^{1/2}.
\end{align}
\end{enumerate}
\end{lemma}

\begin{proof}
Observe that $f$ is descent along the iterates of Algorithm~\ref{alg:NCG-pd}, which implies that $f(x^k) \le f(x^0)$ and hence $\|\nabla f(x^k)\|\le U_g$ due to \eqref{lwbd-Hgupbd}. In addition, since $d^k$ results from Algorithm~\ref{alg:capped-CG} with d$\_$type=SOL, one can see that $\|\nabla f(x^k)\|>\epsilon_g$ and \eqref{SOL-ppty-1}-\eqref{SOL-ppty-3} hold for $d^k$. Moreover, by $\|\nabla f(x^k)\|>\epsilon_g$ and \eqref{SOL-ppty-3}, one can conclude that $d^k\neq 0$.
	
We first prove statement (i). If $\alpha_k=1$, then $\alpha_k\ge \min\{1,(1-\eta)\theta(\epsilon_g/U_g)^{1/2}/2\}$ clearly holds. We now suppose that $\alpha_k < 1$. 
Claim that for all $j\ge 0$ that violate \eqref{ls-sol-stepsize}, it holds that
\begin{equation}\label{claim:lbd-stepsize}
\theta^{(1+\nu)j} \ge (1 - \eta)(\gamma_\nu(\epsilon_g)\epsilon_g)^{1/2}/(H_\nu\|d^k\|^{\nu}).
\end{equation}
Indeed, suppose that \eqref{ls-sol-stepsize} is violated by some $j\ge0$. We next prove that \eqref{claim:lbd-stepsize} holds for such $j$ by considering two separate cases below.

Case 1) $f(x^k+\theta^j d^k)>f(x^k)$. Denote $\varphi(\alpha)=f(x^k+\alpha d^k)$. Then $\varphi(\theta^j)>\varphi(0)$. Also, since $d^k\neq 0$, one has
\[
\varphi^\prime(0) = \nabla f(x^k)^Td^k \overset{\eqref{SOL-ppty-2}}{=} - (d^k)^T(\nabla^2 f(x^k)+2(\gamma_\nu(\epsilon_g)\epsilon_g)^{1/2}I)d^k\overset{\eqref{SOL-ppty-1}}{\le} -(\gamma_\nu(\epsilon_g)\epsilon_g)^{1/2}\|d^k\|^2<0.
\]
Using these, we can observe that there exists a local minimizer $\alpha_*\in(0,\theta^j)$ of $\varphi$ such that $\varphi^\prime(\alpha_*)=\nabla f(x^k+\alpha_*d^k)^Td^k=0$ and $\varphi(\alpha_*)<\varphi(0)$. The latter relation together with $f(x^k) \le f(x^0)$ implies that $f(x^k+\alpha_*d^k)<f(x^k)\le f(x^0)$. Hence, \eqref{apx-nxt-grad} holds for $x=x^k$ and $y=x^k+\alpha_*d^k$. Using this, \eqref{SOL-ppty-1}, \eqref{SOL-ppty-2},  $0<\alpha_*<\theta^j\le 1$, and $\nabla f(x^k+\alpha_*d^k)^Td^k=0$, we deduce that
\begin{equation*}
\begin{array}{l}
\frac{\alpha_*^{1+\nu}H_\nu}{1+\nu}\|d^k\|^{2+\nu}\overset{\eqref{apx-nxt-grad}}{\ge}\|d^k\|\|\nabla f(x^{k}+\alpha_*d^k)-\nabla f(x^k)-\alpha_*\nabla^2 f(x^k)d^k\|\\[6pt]
\ge(d^k)^T(\nabla f(x^k+\alpha_*d^k)-\nabla f(x^k)-\alpha_*\nabla^2 f(x^k)d^k)=-(d^k)^T\nabla f(x^k)-\alpha_*(d^k)^T\nabla^2 f(x^k)d^k\\[1pt]
\overset{\eqref{SOL-ppty-2}}{=}(1-\alpha_*)(d^k)^T(\nabla^2 f(x^k)+2(\gamma_\nu(\epsilon_g)\epsilon_g)^{1/2}I)d^k +2 \alpha_*(\gamma_\nu(\epsilon_g)\epsilon_g)^{1/2}\|d^k\|^2\\
\overset{\eqref{SOL-ppty-1}}{\ge}(1+\alpha_*)(\gamma_\nu(\epsilon_g)\epsilon_g)^{1/2}\|d^k\|^2\ge (\gamma_\nu(\epsilon_g)\epsilon_g)^{1/2}\|d^k\|^2.
\end{array}
\end{equation*}
Recall that $\|d^k\|\neq 0$. Dividing both sides of the last inequality by $H_\nu\|d^k\|^{2+\nu}/(1+\nu)$, we obtain that 
\begin{align*}
\alpha_*^{1+\nu}\ge (1+\nu)(\gamma_\nu(\epsilon_g)\epsilon_g)^{1/2}/(H_\nu\|d^k\|^{\nu}) \ge (\gamma_\nu(\epsilon_g)\epsilon_g)^{1/2}/(H_\nu\|d^k\|^{\nu}),
\end{align*}
which together with $\eta\in(0,1)$ and $\theta^j>\alpha_*$ implies that \eqref{claim:lbd-stepsize} holds in this case.

Case 2) $f(x^k+\theta^j d^k)\le f(x^k)$. This together with $f(x^k)\le f(x^0)$ implies that \eqref{desc-ineq} holds for $x=x^k$ and $y=x^k + \theta^j d^k$. By this and the supposition that $j$ violates \eqref{ls-sol-stepsize}, we obtain that
\begin{align}
&-\eta(\gamma_\nu(\epsilon_g)\epsilon_g)^{1/2}\theta^{2j}\|d^k\|^2\le f(x^k+\theta^j d^k) - f(x^k)\nn\\
&\overset{\eqref{desc-ineq}}{\le}  \theta^j \nabla f(x^k)^Td^k + \frac{\theta^{2j}}{2}(d^k)^T\nabla^2 f(x^k) d^k + \frac{H_{\nu}\theta^{(2+\nu)j}\|d^k\|^{2+\nu}}{(1+\nu)(2+\nu)}\nn\\
& \overset{\eqref{SOL-ppty-2}}{=} -\theta^j (d^k)^T(\nabla^2 f(x^k)+2(\gamma_\nu(\epsilon_g)\epsilon_g)^{1/2}I)d^k + \frac{\theta^{2j}}{2} (d^k)^T\nabla^2 f(x^k) d^k +  \frac{H_{\nu}\theta^{(2+\nu)j}\|d^k\|^{2+\nu}}{(1+\nu)(2+\nu)}\nn\\
&  = -\theta^j \left(1- \frac{\theta^j}{2}\right)(d^k)^T(\nabla^2 f(x^k)+2(\gamma_\nu(\epsilon_g)\epsilon_g)^{1/2}I)d^k - \theta^{2j} (\gamma_\nu(\epsilon_g)\epsilon_g)^{1/2}\|d^k\|^2 + \frac{H_{\nu}\theta^{(2+\nu)j}\|d^k\|^{2+\nu}}{(1+\nu)(2+\nu)}\nn\\
& \overset{\eqref{SOL-ppty-1}}{\le} -\theta^j(\gamma_\nu(\epsilon_g)\epsilon_g)^{1/2}\|d^k\|^2 + \frac{H_{\nu}\theta^{(2+\nu)j}\|d^k\|^{2+\nu}}{(1+\nu)(2+\nu)}.\label{pd-sol-upbd}
\end{align}
Recall that $d^k\neq 0$. Dividing both sides of \eqref{pd-sol-upbd} by $H_\nu\theta^j\|d^k\|^{2+\nu}/[(1+\nu)(2+\nu)]$, and using $\theta,\eta\in(0,1)$ and $\nu\in[0,1]$, we obtain that
\begin{align*}
\theta^{(1+\nu)j} \ge (1+\nu)(2+\nu) (1 - \theta^j\eta)(\gamma_\nu(\epsilon_g)\epsilon_g)^{1/2}/(H_\nu\|d^k\|^{\nu}) \ge (1 - \eta)(\gamma_\nu(\epsilon_g)\epsilon_g)^{1/2}/(H_\nu\|d^k\|^{\nu}).
\end{align*}
Hence, \eqref{claim:lbd-stepsize} also holds in this case.

Combining the above two cases, we conclude that \eqref{claim:lbd-stepsize} holds for all $j\ge0$ violating \eqref{ls-sol-stepsize}. Letting $\gamma=\gamma_\nu(\epsilon_g)$ in Lemma~\ref{lem:tech-gammamu-eps}, we know from \eqref{rela:Hnu-gammanu-1} that $(\gamma_\nu(\epsilon_g)\epsilon_g)^{1/2}/H_\nu \ge 2^{1+\nu}(\epsilon_g/\gamma_\nu(\epsilon_g))^{\nu/2}$, which along with \eqref{claim:lbd-stepsize} implies that
\begin{align}\label{a-handy-relation-sol}
\theta^{(1+\nu)j} \ge (1-\eta) 2^{1+\nu} (\epsilon_g/\gamma_\nu(\epsilon_g))^{\nu/2}/\|d^k\|^\nu.
\end{align}
In addition, notice from Algorithm~\ref{alg:NCG-pd} that when $\alpha_k<1$, at least one of $f(x^k + d^k) \le f(x^k)$ and $\|\nabla f(x^k+d^k)\|\le\epsilon_g$ does not hold. It then follows from Lemma~\ref{lem:next-FOSP-or-large-step} that $6\|d^k\|\ge(\epsilon_g/\gamma_\nu(\epsilon_g))^{1/2}$.  Using this and taking the $\left(\frac{1}{1+\nu}\right)$th root of both sides of \eqref{a-handy-relation-sol}, we deduce that
\begin{align}
\theta^j & \ge 2(1-\eta)^{1/(1+\nu)}\left(\frac{(\epsilon_g/\gamma_\nu(\epsilon_g))^{1/4}}{\|d^k\|^{1/2}}\right)^{2\nu/(1+\nu)}\nonumber\\
& \ge 2 (1-\eta)\left(\frac{(\epsilon_g/\gamma_\nu(\epsilon_g))^{1/4}}{\|d^k\|^{1/2}}\right)^{2\nu/(1+\nu)} \left(\frac{(\epsilon_g/\gamma_\nu(\epsilon_g))^{1/4}}{\sqrt{6}\|d^k\|^{1/2}}\right)^{(1-\nu)/(1+\nu)} \ge \frac{2(1-\eta)(\epsilon_g/\gamma_\nu(\epsilon_g))^{1/4}}{3\|d^k\|^{1/2}}.\label{claim:lbd-stepsize-1}
\end{align}
By this and $\theta\in(0,1)$, one can see that all $j\ge0$ that violate \eqref{ls-sol-stepsize} must be bounded above. It then follows that the step length $\alpha_k$ associated with \eqref{ls-sol-stepsize} is well-defined. We next derive a lower bound for $\alpha_k$. Notice from the definition of $j_k$ in Algorithm~\ref{alg:NCG-pd} that $j=j_k-1$ violates \eqref{ls-sol-stepsize} and hence \eqref{claim:lbd-stepsize-1} holds for $j=j_k-1$. Then, by \eqref{claim:lbd-stepsize-1} with $j=j_k-1$ and $\alpha_k=\theta^{j_k}$, one has 
\begin{align}\label{lbd-step-size-sol}
\alpha_k = \theta^{j_k} \ge 2(1-\eta)\theta(\epsilon_g/\gamma_\nu(\epsilon_g))^{1/4}/(3\|d^k\|^{1/2}),
\end{align}
which together with \eqref{SOL-ppty-upbd} and $\|\nabla f(x^k)\|\le U_g$ implies
\[
\alpha_k \ge 2(1-\eta)\theta \epsilon_g^{1/2}/(3\sqrt{1.1}\|\nabla f(x^k)\|^{1/2}) \ge (1-\eta)\theta(\epsilon_g/U_g)^{1/2}/2,
\]
and hence statement (i) holds.

We next prove statement (ii). To prove this, it suffices to show that \eqref{ineq:descent-sol-pd} holds under the condition $\|\nabla f(x^{k+1})\|>\epsilon_g$. To this end, we assume that $\|\nabla f(x^{k+1})\|>\epsilon_g$ holds, and prove \eqref{ineq:descent-sol-pd} by considering two separate cases below.

Case 1) $\alpha_k=1$. By this and the assumption $\|\nabla f(x^{k+1})\|>\epsilon_g$, one can observe from Algorithm~\ref{alg:NCG-pd} that \eqref{ls-sol-stepsize} holds for $j=0$. It then follows that $f(x^k+d^k)\le f(x^k)\le f(x^0)$, which implies that \eqref{i-HL-1} holds for $x=x^k$ and $y=x^k+d^k$. By this, $\alpha_k=1$, $\|\nabla f(x^{k+1})\|>\epsilon_g$, \eqref{i-HL-1}, and \eqref{SOL-ppty-3}, one has
\begin{align*}
\epsilon_g <\|\nabla f(x^{k+1})\| = \|\nabla f(x^k + d^k)\|&\le \|\nabla f(x^k+d^k)-\nabla f(x^k)-\nabla^2 f(x^k)d^k\|\\
&\qquad +\|(\nabla^2 f(x^k)+2(\gamma_\nu(\epsilon_g)\epsilon_g)^{1/2}I)d^k+\nabla f(x^k)\|+2(\gamma_\nu(\epsilon_g)\epsilon_g)^{1/2}\|d^k\| \\
&\overset{\eqref{i-HL-1}\eqref{SOL-ppty-3}}\le \frac{L(\epsilon_g/2)}{2}\|d^k\|^2 + \frac{\epsilon_g}{2} + \frac{4+\zeta}{2}(\gamma_\nu(\epsilon_g)\epsilon_g)^{1/2}\|d^k\|.
\end{align*}
 Solving the above inequality for $\|d^k\|$ and using $\|d^k\|>0$, we obtain that
\begin{align*}
&\|d^k\| \ge \frac{-(4+\zeta)(\gamma_\nu(\epsilon_g)\epsilon_g)^{1/2} + \sqrt{(4+\zeta)^2\gamma_\nu(\epsilon_g)\epsilon_g + 4 L(\epsilon_g/2)\epsilon_g}}{2L(\epsilon_g/2)}\\
&= \frac{2\epsilon_g}{(4+\zeta)(\gamma_\nu(\epsilon_g)\epsilon_g)^{1/2} + \sqrt{(4+\zeta)^2\gamma_\nu(\epsilon_g)\epsilon_g + 4L(\epsilon_g/2)\epsilon_g}} \ge \frac{2}{4+\zeta + \sqrt{(4+\zeta)^2 + 1}}\left(\frac{\epsilon_g}{\gamma_\nu(\epsilon_g)}\right)^{1/2},
\end{align*}
where the last inequality is due to $L(\epsilon_g/2)\le\gamma_\nu(\epsilon_g)/4$ (see \eqref{L1-bd-gma} with $a=2$). By the above inequality, $\alpha_k=1$, and \eqref{ls-sol-stepsize}, one has 
\[
f(x^k) - f(x^{k+1}) \ge \eta (\gamma_\nu(\epsilon_g)\epsilon_g)^{1/2}\|d^k\|^2\ge \eta \left(\frac{2}{4+\zeta + \sqrt{(4+\zeta)^2 + 1}}\right)^2 \frac{\epsilon_g^{3/2}}{\gamma_\nu(\epsilon_g)^{1/2}},
\]
and hence statement (i) holds in this case.
    
Case 2) $\alpha_k<1$.  By this,  one can observe from Algorithm~\ref{alg:NCG-pd} that at least one of $\|\nabla f(x^k + d^k)\|\le\epsilon_g$ and $f(x^k+d^k)\le f(x^k)$ does not hold. It then follows from Lemma~\ref{lem:next-FOSP-or-large-step} that $6\|d^k\| \ge (\epsilon_g/\gamma_\nu(\epsilon_g))^{1/2}$. Using this, \eqref{ls-sol-stepsize}, and \eqref{lbd-step-size-sol}, we can deduce that
\begin{align*}
f(x^k) - f(x^{k+1}) \overset{\eqref{ls-sol-stepsize}}{>} \eta (\gamma_\nu(\epsilon_g)\epsilon_g)^{1/2} \theta^{2j_k}\|d^k\|^2 \overset{\eqref{lbd-step-size-sol}}{\ge} \eta\left(\frac{2(1-\eta)\theta}{3}\right)^2 \epsilon_g\|d^k\| \ge \frac{\eta}{6} \left(\frac{2(1-\eta)\theta}{3}\right)^2 \frac{\epsilon_g^{3/2}}{\gamma_\nu(\epsilon_g)^{1/2}},
\end{align*}
where the last inequality is due to $6\|d^k\| \ge (\epsilon_g/\gamma_\nu(\epsilon_g))^{1/2}$. By the above inequality and the definition of $c_{\mathrm{sol}}$ in \eqref{csol-cnc}, one can see that \eqref{ineq:descent-sol-pd} also holds in this case. 
\end{proof}

The following lemma shows when the search direction $d^k$ in Algorithm~\ref{alg:NCG-pd} is of type `NC', the line search step results in a sufficient reduction on $f$ as well.

\begin{lemma}\label{lem:nc-1}
Suppose that Assumption~\ref{asp:NCG-cmplxity} holds and the direction $d^k$ results from Algorithm~\ref{alg:capped-CG} with d$\_$type=NC at some iteration $k$ of Algorithm~\ref{alg:NCG-pd}. Let $\gamma_\nu(\epsilon_g)$ and $c_{\mathrm{nc}}$ be defined in \eqref{gma-eps} and \eqref{csol-cnc}, respectively. Then the following statements hold.
\begin{enumerate}[{\rm (i)}]
\item The step length $\alpha_k$ is well-defined, and $\alpha_k\ge\theta\min\{1,1/\gamma_\nu(\epsilon_g)\}$.
\item The next iterate $x^{k+1}=x^k + \alpha_k d^k$ satisfies $f(x^k) - f(x^{k+1}) \ge c_{\mathrm{nc}}\epsilon_g^{3/2}/\gamma_\nu(\epsilon_g)^{1/2}$.
\end{enumerate}
\end{lemma}

\begin{proof}
Observe that $f$ is descent along the iterates of Algorithm~\ref{alg:NCG-pd} and hence $f(x^k) \le f(x^0)$. Since $d^k$ results from Algorithm~\ref{alg:capped-CG} with d$\_$type=NC, one can see from Lemma~\ref{lem:SOL-NC-ppty}(ii) that
\begin{equation}\label{12-nc}
\nabla f(x^k)^Td^k\le 0,\quad (d^k)^T\nabla^2 f(x^k) d^k/\|d^k\|^2 = -\min\{1,\gamma_\nu(\epsilon_g)\}\|d^k\| \le-(\gamma_{\nu}(\epsilon_g)\epsilon_g)^{1/2}<0.
\end{equation}

We first prove statement (i). If \eqref{ls-nc-stepsize} holds for $j=0$, then $\alpha_k=1$, which together with $\theta\in(0,1)$ implies that $\alpha_k\ge\theta\min\{1,1/\gamma_\nu(\epsilon_g)\}$ holds. We now suppose that \eqref{ls-nc-stepsize} fails for $j=0$. Claim that for all $j\ge 0$ that violate \eqref{ls-nc-stepsize}, it holds that
\begin{align}\label{lwbd-j-violate-ls-nc}
\theta^{\nu j} > (1-\eta/2)\min\{1,\gamma_\nu(\epsilon_g)\}^\nu(\gamma_\nu(\epsilon_g)\epsilon_g)^{(1-\nu)/2}/H_\nu.
\end{align}
Indeed, suppose that \eqref{ls-nc-stepsize} is violated by some $j\ge0$. We now show that \eqref{lwbd-j-violate-ls-nc} holds for such $j$ by considering two separate cases below.

Case 1) $f(x^k + \theta^jd^k) > f(x^k)$. Let $\varphi(\alpha) = f(x^k + \alpha d^k)$. Then $\varphi(\theta^j) > \varphi(0)$. Also, by \eqref{12-nc}, one has 
\begin{align*}
\varphi^\prime(0) = \nabla f(x^k)^T d^k\le 0,\quad \varphi^{\prime\prime}(0) = (d^k)^T\nabla^2 f(x^k) d^k < 0.
\end{align*}
By these and $\varphi(\theta^j)>\varphi(0)$, it is not hard to observe that there exists a local minimizer $\alpha_*\in(0,\theta^j)$ of $\varphi$ such that $\varphi(\alpha_*)<\varphi(0)$, namely, $f(x^k+\alpha_*d^k)< f(x^k)$. Further, by the second-order optimality condition of $\varphi$ at $\alpha_*$, one has $\varphi^{\prime\prime}(\alpha_*)=(d^k)^Tf(x^k+\alpha_*d^k)d^k\ge0$. Since $f(x^k+\alpha_*d^k)< f(x^k)\le f(x^0)$, it follows that \eqref{F-Hess-Lip} holds for $x=x^k$ and $y=x^k+\alpha_*d^k$. Using this, the second relation in \eqref{12-nc}, and $(d^k)^T\nabla^2 f(x^k+\alpha_*d^k)d^k\ge0$, we obtain that 
\begin{align}
H_\nu \alpha_*^\nu\|d^k\|^{2+\nu}&\ge \|d^k\|^2\|\nabla^2 f(x^k+\alpha_*d^k)-\nabla^2 f(x^k)\|\ge (d^k)^T(\nabla^2 f(x^k+\alpha_*d^k)-\nabla^2 f(x^k))d^k\nonumber\\
&\ge - (d^k)^T\nabla^2 f(x^k) d^k = \min\{1,\gamma_\nu(\epsilon_g)\}\|d^k\|^3.\label{neig-nc}
\end{align}
Recall from \eqref{12-nc} that $\|d^k\|\ge\max\{1,1/\gamma_\nu(\epsilon_g)\}(\gamma_{\nu}(\epsilon_g)\epsilon_g)^{1/2}>0$. Using this, $\theta^j>\alpha_*$, and \eqref{neig-nc}, we deduce that
\begin{align*}
\theta^{\nu j} \ge  \alpha_*^\nu \overset{\eqref{neig-nc}}{\ge} \min\{1,\gamma_\nu(\epsilon_g)\} \|d^k\|^{1-\nu}/H_\nu\ge \min\{1,\gamma_\nu(\epsilon_g)\}^\nu (\gamma_{\nu}(\epsilon_g)\epsilon_g)^{(1-\nu)/2}/H_\nu,    
\end{align*}
which together with $\eta\in(0,1)$ implies that \eqref{lwbd-j-violate-ls-nc} holds in this case.

Case 2) $f(x^k + \theta^j d^k)\le f(x^k)$. This and $f(x^k) \le f(x^0)$ imply that $f(x^k + \theta^j d^k)\le f(x^k)  \le f(x^0)$. It then follows that  \eqref{desc-ineq} holds for $x=x^k$ and $y=x^k + \theta^j d^k$. By this, \eqref{12-nc}, and the supposition that $j$ violates \eqref{ls-nc-stepsize}, one has
\begin{align*}
-\frac{\eta}{4}\min\{1,\gamma_\nu(\epsilon_g)\}\theta^{2j}\|d^k\|^3& < f(x^k + \theta^j d^k) - f(x^k)\\
&\overset{\eqref{desc-ineq}}{\le} \theta^j \nabla f(x^k)^T d^k + \frac{\theta^{2j}}{2} (d^k)^T\nabla^2 f(x^k) d^k + \frac{H_\nu \theta^{(2+\nu)j}}{(1+\nu)(2+\nu)} \|d^k\|^{2+\nu}\\
&\overset{\eqref{12-nc}}{\le} -\frac{1}{2}\min\{1,\gamma_\nu(\epsilon_g)\}\theta^{2j}\|d^k\|^3 + \frac{H_\nu \theta^{(2+\nu)j}}{(1+\nu)(2+\nu)} \|d^k\|^{2+\nu},
\end{align*}
which together with $\|d^k\|\ge\max\{1,1/\gamma_\nu(\epsilon_g)\}(\gamma_\nu(\epsilon_g)\epsilon_g)^{1/2}>0$, $\eta\in(0,1)$, and $\nu\in[0,1]$ implies that
\[
\theta^{\nu j} > (1+\nu)(2+\nu)(1/2-\eta/4)\min\{1,\gamma_\nu(\epsilon_g)\}\|d^k\|^{1-\nu}/H_\nu \ge (1-\eta/2)\min\{1,\gamma_\nu(\epsilon_g)\}^\nu(\gamma_\nu(\epsilon_g)\epsilon_g)^{(1-\nu)/2}/H_\nu,
\]
and hence \eqref{lwbd-j-violate-ls-nc} also holds in this case.



Combining the above two cases, we conclude that \eqref{lwbd-j-violate-ls-nc} holds for any $j\ge 0$ violating \eqref{ls-nc-stepsize}. Letting $\gamma=\gamma_\nu(\epsilon_g)$ in Lemma~\ref{lem:tech-gammamu-eps}, we obtain from \eqref{rela:Hnu-gammanu-2} that $(\gamma_\nu(\epsilon_g)\epsilon_g)^{(1-\nu)/2}/H_\nu \ge 2^{1+\nu}/\gamma_\nu(\epsilon_g)^\nu$, which together with \eqref{lwbd-j-violate-ls-nc} and $\eta\in(0,1)$ implies that for any $j\ge 0$ violating \eqref{ls-nc-stepsize}, 
\begin{align}\label{a-counter-ineq}
\theta^{\nu j} > (1-\eta/2) \min\{1,\gamma_\nu(\epsilon_g)\}^\nu 2^{1+\nu}/\gamma_\nu(\epsilon_g)^\nu > 2^\nu \min\{1,\gamma_\nu(\epsilon_g)\}^\nu/\gamma_\nu(\epsilon_g)^\nu.    
\end{align}
When $\nu=0$, one can see that \eqref{a-counter-ineq} does not hold for $j=0$, which implies that \eqref{ls-nc-stepsize} holds for $j=0$  and hence $\alpha_k=1$. When $\nu\in(0,1]$, one can see from \eqref{a-counter-ineq} that  all $j\ge 0$ violating \eqref{ls-nc-stepsize} must be bounded above, and hence $\alpha_k$ is well-defined. We next derive a lower bound for $\alpha_k$. Notice from the definition of $j_k$ that $j=j_k-1$ violates \eqref{ls-nc-stepsize}. Then, using \eqref{a-counter-ineq} with $j=j_k-1$ and $\alpha_k=\theta^{j_k}$, we see that $\alpha_k=\theta^{j_k}\ge \theta\min\{1,1/\gamma_\nu(\epsilon_g)\}$. Hence, statement (i) holds.


We next prove statement (ii). Recall from \eqref{12-nc} that $\|d^k\|\ge\max\{1,1/\gamma_\nu(\epsilon_g)\}(\gamma_\nu(\epsilon_g)\epsilon_g)^{1/2}$. It then follows from this, $\alpha_k\ge \theta\min\{1,1/\gamma_\nu(\epsilon_g)\}$, and \eqref{ls-nc-stepsize} that
\begin{align*}
f(x^k) - f(x^{k+1}) & \geq \frac{\eta}{4} \min\{1,\gamma_\nu(\epsilon_g)\}\alpha_k^2\|d^k\|^3 \\
&= \frac{\eta}{4}\theta^2 \min\{1,\gamma_\nu(\epsilon_g)\}\min\{1,1/\gamma_\nu(\epsilon_g)\}^2\max\{1,1/\gamma_\nu(\epsilon_g)\}^3 (\gamma_\nu(\epsilon_g)\epsilon_g)^{3/2}\\
&= \frac{\eta}{4}\theta^2 \min\{1,\gamma_\nu(\epsilon_g)\}^3\max\{1,1/\gamma_\nu(\epsilon_g)\}^3 \epsilon_g^{3/2}/\gamma_\nu(\epsilon_g)^{1/2} =  \frac{\eta}{4}\theta^2 \epsilon_g^{3/2}/\gamma_\nu(\epsilon_g)^{1/2}.    
\end{align*}
This together with the definition of $c_{\mathrm{nc}}$ in \eqref{csol-cnc} implies that statement (ii) holds.
\end{proof}

We are now ready to prove Theorem~\ref{thm:c-pd}.

\begin{proof}[Proof of Theorem~\ref{thm:c-pd}]	
(i) Recall from the assumption of this theorem that $\epsilon_H$ is not provided for Algorithm~\ref{alg:NCG-pd}. It then follows that Algorithm~\ref{alg:capped-CG} is called at each iteration of Algorithm~\ref{alg:NCG-pd}, except the last iteration. Suppose for contradiction that the total number of iterations of Algorithm~\ref{alg:NCG-pd} is more than $K_1$. Observe from Algorithm~\ref{alg:NCG-pd} and Lemmas~\ref{lem:sol-1}(ii) and \ref{lem:nc-1}(ii) that each call of Algorithm~\ref{alg:capped-CG} results in a reduction on $f$ at least by $\min\{c_{\mathrm{sol}},c_{\mathrm{nc}}\}\epsilon_g^{3/2}/\gamma_\nu(\epsilon_g)^{1/2}$. Using this and \eqref{lwbd-Hgupbd}, we have
\begin{equation*}
K_1\min\{c_{\mathrm{sol}},c_{\mathrm{nc}}\}\epsilon_g^{3/2}/\gamma_\nu(\epsilon_g)^{1/2}\le\sum_{k=0}^{K_1-1} (f(x^k) - f(x^{k+1}))=f(x^0) - f(x^{K_1}) \le f(x^0) - \fl,
\end{equation*} 
which contradicts the definition of $K_1$ given in \eqref{K1}. Hence, the total number of iterations of Algorithm~\ref{alg:NCG-pd} is no more than $K_1$. In addition, the relation~\eqref{K1-order} follows from \eqref{gma-eps}, \eqref{K1}, and \eqref{csol-cnc}. Since $\epsilon_H$ is not provided, one can also observe from Algorithm~\ref{alg:NCG-pd} that its output $x^k$ satisfies $\|\nabla f(x^k)\|\le\epsilon_g$ for some $0\le k\le K_1$. This completes the proof of statement (i) of Theorem~\ref{thm:c-pd}.

(ii) Notice that $f$ is descent along the iterates generated by Algorithm~\ref{alg:NCG-pd}, which implies that $f(x^k) \le f(x^0)$ for all $0\leq k \leq K_1$. Using this and \eqref{lwbd-Hgupbd}, we have that $\|\nabla^2 f(x^k)\|\le U_H$ for all $0\leq k \leq K_1$.  By Theorem~\ref{lem:capped-CG} with $(H,\varepsilon)=(\nabla^2 f(x^k),(\gamma_\nu(\epsilon_g)\epsilon_g)^{1/2})$ and the fact that $\|\nabla^2 f(x^k)\|\le U_H$ for all $k$, one can observe that the number of gradient evaluations and Hessian-vector products of $f$ required by each call of Algorithm~\ref{alg:capped-CG} with $U=0$ in Algorithm~\ref{alg:NCG-pd} is at most $\widetilde{\cO}(\min\{n,U_H^{1/2}/(\gamma_\nu(\epsilon_g)\epsilon_g)^{1/4}\})$. In view of this and statement (i), we see that statement (ii) of 
Theorem~\ref{thm:c-pd} holds .
\end{proof}

The following lemma shows that when the search direction $d^k$ in Algorithm~\ref{alg:NCG-pd} is a negative curvature direction returned from Algorithm~\ref{pro:meo}, the next iterate $x^{k+1}$ produces a sufficient reduction on $f$.

\begin{lemma}\label{lem:meo-dec}
Suppose that Assumption~\ref{asp:NCG-cmplxity} holds with $\nu\in(0,1]$, and the direction $d^k$ results from Algorithm~\ref{pro:meo} at some iteration $k$ of Algorithm~\ref{alg:NCG-pd}. Let $c_{\mathrm{meo}}$ be defined in \eqref{cmeo}. Then the following statements hold.
\begin{enumerate}[{\rm (i)}]
\item The step length $\alpha_k$ is well-defined, and $\alpha_k\ge\min\{1,\theta((1-\eta)/H_\nu)^{1/\nu}(\epsilon_H/2)^{(1-\nu)/\nu}\}$.
\item The next iterate $x^{k+1}=x^k+\alpha_k d^k$ satisfies $f(x^k) - f(x^{k+1}) \ge c_{\mathrm{meo}} \epsilon_H^{(2+\nu)/\nu}$.
\end{enumerate}
\end{lemma}

\begin{proof}
Observe that $f$ is descent along the iterates generated by Algorithm~\ref{alg:NCG-pd}, which implies that $f(x^k)\le f(x^0)$. Since $d^k$ results from Algorithm~\ref{pro:meo}, it follows from Algorithm~\ref{alg:NCG-pd} that $d^k$ is given in \eqref{dk-nc-meo2} with $v$ being the vector returned from Algorithm~\ref{pro:meo} with $H=\nabla^2 f(x^k)$ and $\varepsilon=\epsilon_H$ that satisfies $\|v\|=1$ and $v^T\nabla^2 f(x^k)v\leq -\epsilon_H/2$. By these and \eqref{dk-nc-meo2}, one can see that
\begin{equation}\label{ppty-nc-meo}
\nabla f(x^k)^Td^k\le 0,\quad (d^k)^T\nabla^2 f(x^k)d^k/\|d^k\|^2 = - \|d^k\| =v^T\nabla^2 f(x^k)v\le -\epsilon_H/2<0.
\end{equation}

We first prove statement (i). If \eqref{ls-meo-pd} holds for $j=0$, then we have $\alpha_k=1$, which clearly implies that $\alpha_k\ge\min\{1,\theta[(1-\eta)/H_\nu]^{1/\nu}(\epsilon_H/2)^{(1-\nu)/\nu}\}$. We now suppose that \eqref{ls-meo-pd} fails for $j=0$. Claim that for all $j\ge0$ that violate \eqref{ls-meo-pd}, it holds that
\begin{align}\label{claim:meo-vio-lbd}
\theta^j \ge ((1-\eta)/H_\nu)^{1/{\nu}}(\epsilon_H/2)^{(1-\nu)/{\nu}}.    
\end{align}
Indeed, suppose that \eqref{ls-meo-pd} is violated by some $j\ge 0$. We now show that \eqref{claim:meo-vio-lbd} holds for such $j$ by considering two separate cases below.

Case 1) $f(x^k + \theta^j d^k)> f(x^k)$. Let $\varphi(\alpha)=f(x^k + \alpha d^k)$. Then $\varphi(\theta^j)>\varphi(0)$. Also, by \eqref{ppty-nc-meo}, one has 
\[
\varphi^\prime(0) = \nabla f(x^k)^T d^k\le0,\quad \varphi^{\prime\prime}(0) = (d^k)^T\nabla^2 f(x^k) d^k<0.
\]
By these and $\varphi(\theta^j)>\varphi(0)$, it is not hard to observe that there exists a local minimizer $\alpha_*\in(0,\theta^j)$ of $\varphi$ such that $\varphi(\alpha_*)<\varphi(0)$, namely, $f(x^k+\theta^j d^k)< f(x^k)$. By the second-order optimality condition of $\varphi$, one has $\varphi^{\prime\prime}(\alpha_*)=(d^k)^T\nabla^2 f(x^k + \alpha_* d^k) d^k\ge0$. Since $f(x^k + \alpha_* d^k)\le f(x^k) \le f(x^0)$, it follows that \eqref{F-Hess-Lip} holds for $x=x^k$ and $y=x^k+\alpha_*d^k$. Using this, the second relation in \eqref{ppty-nc-meo}, and $(d^k)^T\nabla^2 f(x^k + \alpha_* d^k) d^k\ge0$, we obtain that
\begin{align*}
H_\nu \alpha_*^\nu\|d^k\|^{2+\nu} \ge &\ \|d^k\|^2\|\nabla^2 f(x^k+\alpha_* d^k)-\nabla^2 f(x^k)\|\ge (d^k)^T(\nabla^2 f(x^k+\alpha_* d^k)-\nabla^2 f(x^k))d^k\\
\ge&\ -(d^k)^T\nabla^2 f(x^k) d^k = \|d^k\|^3.
\end{align*}
Recall from \eqref{ppty-nc-meo} that $d^k\neq 0$. Dividing both sides of this inequality by $H_\nu \|d^k\|^{2+\nu} $ yields $\alpha_*^\nu \ge \|d^k\|^{1-\nu}/H_\nu$, which along with $\theta^j> \alpha_*$, $\nu\in(0,1]$, and $\|d^k\|\ge \epsilon_H/2$ implies that $\theta^j \ge (\epsilon_H/2)^{(1-\nu)/\nu}/H_\nu^{1/\nu}$ and  hence \eqref{claim:meo-vio-lbd} holds in this case.

Case 2) $f(x^k + \theta^j d^k)\le f(x^k)$. This and $f(x^k) \le f(x^0)$ imply that $f(x^k + \theta^j d^k)\le f(x^k)  \le f(x^0)$. It then follows that \eqref{desc-ineq} holds for $x=x^k$ and $y=x^k+\theta^j d^k$. By this and the supposition that $j$ violates \eqref{ls-meo-pd}, one has 
\begin{align*}
-\frac{\eta}{2}\theta^{2j}\|d^k\|^3 \le &\ f(x^k+\theta^j d^k) - f(x^k) \overset{\eqref{desc-ineq}}{\le} \theta^j\nabla f(x^k)^Td^k + \frac{\theta^{2j}}{2}(d^k)^T\nabla^2 f(x^k) d^k + \frac{H_\nu\theta^{(2+\nu)j}}{(1+\nu)(2+\nu)}\|d^k\|^{2+\nu}\\
\overset{\eqref{ppty-nc-meo}}{\le}&\ -\frac{\theta^{2j}}{2}\|d^k\|^3 + \frac{H_\nu\theta^{(2+\nu)j}}{2}\|d^k\|^{2+\nu},
\end{align*}
where the last inequality is due to $\nu\in(0,1]$ and \eqref{ppty-nc-meo}. By this and $d^k\neq 0$, one has $\theta^{\nu j}\ge (1-\eta) \|d^k\|^{1-\nu}/H_\nu$. This together with $\nu\neq0$ implies that \eqref{claim:meo-vio-lbd} holds in this case as well.

Combining the above two cases, we conclude that \eqref{claim:meo-vio-lbd} holds for all $j\ge0$ that violate \eqref{ls-meo-pd}. By this and $\theta\in(0,1)$, one can see that all $j\ge0$ violating \eqref{ls-meo-pd} must be bounded above. It then follows that the step length $\alpha_k$ associated with \eqref{ls-meo-pd} is well-defined. We next derive a lower bound for $\alpha_k$. Notice from the definition of $j_k$ in Algorithm~\ref{alg:NCG-pd} that $j=j_k-1$ violates \eqref{ls-meo-pd} and hence \eqref{claim:meo-vio-lbd} holds for $j=j_k-1$. Then, by \eqref{claim:meo-vio-lbd} with $j=j_k-1$ and $\alpha_k=\theta^{j_k}$, one has $\alpha_k=\theta^{j_k}\ge \theta[(1-\eta)/H_\nu]^{1/\nu}(\epsilon_H/2)^{(1-\nu)/\nu}$. Hence, $\alpha_k\ge\min\{1,\theta[(1-\eta)/H_\nu]^{1/\nu}(\epsilon_H/2)^{(1-\nu)/\nu}\}$ holds as desired.

We next prove statement (ii). Recall from \eqref{ppty-nc-meo} that $\|d^k\|\ge\epsilon_H/2$. In view of this, \eqref{ls-meo-pd}, and the fact that $\alpha_k\ge\min\{1,\theta[(1-\eta)/H_\nu]^{1/\nu}(\epsilon_H/2)^{(1-\nu)/\nu}\}$, we obtain that
\begin{align*}
f(x^k) - f(x^{k+1}) & > \frac{\eta}{2} \alpha_k^2\|d^k\|^3 \ge \frac{\eta}{2}\min\big\{1,\theta((1-\eta)/H_\nu)^{1/\nu}(\epsilon_H/2)^{(1-\nu)/\nu}\big\}^2(\epsilon_H/2)^3\\
&= \frac{\eta}{2}\min\big\{(\epsilon_H/2)^{-(1-\nu)/\nu},\theta((1-\eta)/H_\nu)^{1/\nu}\big\}^2 (\epsilon_H/2)^{(2+\nu)/\nu}\\
&\ge \frac{\eta}{2}\min\big\{1,\theta((1-\eta)/H_\nu)^{1/\nu}\big\}^2 (\epsilon_H/2)^{(2+\nu)/\nu},
\end{align*}
where the last inequality is due to $\epsilon_H\in(0,1)$ and $\nu\in(0,1]$. By this inequality and the definition of $c_{\mathrm{meo}}$ in \eqref{cmeo}, one can see that statement (ii) holds.
\end{proof}

We are now ready to provide a proof of Theorem~\ref{thm:c-pd-sosp}.

\begin{proof}[Proof of Theorem~\ref{thm:c-pd-sosp}]
(i) Let $K_1$ and $K_2$ be defined in \eqref{K1} and \eqref{K2}, respectively. We first claim that the total number of calls of Algorithm~\ref{pro:meo} in Algorithm~\ref{alg:NCG-pd} is at most $K_2$. Indeed, suppose for contradiction that its total number of calls is more than $K_2$. Observe from Algorithm~\ref{alg:NCG-pd} and Lemma~\ref{lem:meo-dec}(ii) that each of these calls, except the last one, results in a reduction on $f$ at least by $c_{\mathrm{meo}}\epsilon_H^{(2+\nu)/\nu}$. Since $f$ is descent along the iterates of Algorithm~\ref{alg:NCG-pd} and $f(x^k) \geq \fl$,  the total amount of reduction on $f$ resulting from the calls of Algorithm~\ref{pro:meo} in Algorithm~\ref{alg:NCG-pd} is at most $f(x^0) - \fl$. Combining these observations, one has
\begin{align*}
K_2c_{\mathrm{meo}}\epsilon_H^{(2+\nu)/\nu} \le f(x^0) - \fl,
\end{align*}
which contradicts the definition of $K_2$ in \eqref{K2}. Hence, the total number of calls of Algorithm~\ref{pro:meo} is at most $K_2$. 

We next claim that the total number of calls of Algorithm~\ref{alg:capped-CG} in Algorithm~\ref{alg:NCG-pd} is at most $K_1+K_2-1$. Indeed, suppose for contradiction that its total number of calls is more than $K_1+K_2-1$. Notice that if Algorithm~\ref{alg:capped-CG} is called at some iteration $k$ and generates $x^{k+1}$ satisfying $\|\nabla f(x^{k+1})\| \le \epsilon_g$, then Algorithm~\ref{pro:meo} must be called at the iteration $k+1$. In view of this and the fact that the total number of calls of Algorithm~\ref{pro:meo} is at most $K_2$, one can observe that the total number of such iterations is at most $K_2$. This along with the above supposition implies that the total number of iterations $k$ of Algorithm~\ref{alg:NCG-pd} at which Algorithm~\ref{alg:capped-CG} is called and generates the next iterate $x^{k+1}$ satisfying $\|\nabla f(x^{k+1})\|>\epsilon_g$ is at least $K_1$. For each of these iterations $k$, we observe from Lemmas~\ref{lem:sol-1}(ii) and \ref{lem:nc-1}(ii) that $f(x^k) - f(x^{k+1}) \ge \min\{c_{\mathrm{sol}},c_{\mathrm{nc}}\}\epsilon_g^{3/2}/\gamma_\nu(\epsilon_g)^{1/2}$.  Since $f$ is descent along the iterates of Algorithm~\ref{alg:NCG-pd} and $f(x^k) \geq \fl$,  the total amount of reduction on $f$ resulting from these iterations $k$ is at most $f(x^0) - \fl$. It then follows that
\begin{align*}
K_1\min\{c_{\mathrm{sol}},c_{\mathrm{nc}}\}\epsilon_g^{3/2}/\gamma_\nu(\epsilon_g)^{1/2} \le f(x^0) - \fl,     
\end{align*}
which contradicts the definition of $K_1$ given in \eqref{K1}. Hence, the total number of calls of Algorithm~\ref{alg:capped-CG} in Algorithm~\ref{alg:NCG-pd} is at most $K_1+K_2-1$. 

Based on the above claims and the fact that either Algorithm~\ref{alg:capped-CG} or \ref{pro:meo} is called at each iteration of Algorithm~\ref{alg:NCG-pd}, we conclude that the total number of iterations of Algorithm~\ref{alg:NCG-pd} is at most $K_1+2K_2 - 1$. In addition, relation~\eqref{K1K2-order} follows from \eqref{gma-eps}, \eqref{csol-cnc},  \eqref{K1},  \eqref{cmeo}, and \eqref{K2}. Moreover, one can observe that the output $x^k$ of Algorithm~\ref{alg:NCG-pd} satisfies $\|\nabla f(x^k)\|\le\epsilon_g$ deterministically and $\lambda_{\min}(\nabla^2 f(x^k)) \ge -\epsilon_H$ with probability at least $1-\delta$ for some $0\le k\le K_1 + 2K_2 -1$, where the latter part is due to Algorithm~\ref{pro:meo}. This completes the proof of statement (i) of Theorem~\ref{thm:c-pd-sosp}.

(ii) By Theorem~\ref{lem:capped-CG} with $(H,\varepsilon)=(\nabla^2 f(x^k),(\gamma_\nu(\epsilon_g)\epsilon_g)^{1/2})$ and the fact that $\|\nabla^2 f(x^k)\|\le U_H$, one can observe that the number of gradient evaluations and Hessian-vector products of $f$ required by each call of Algorithm~\ref{alg:capped-CG} with input $U=0$ is at most $\widetilde{\cO}(\min\{n,U_H^{1/2}/(\gamma_\nu(\epsilon_g)\epsilon_g)^{1/4}\})$. In addition, by Theorem~\ref{rand-Lanczos} with $(H,\varepsilon)=(\nabla^2 f(x^k),\epsilon_H)$, $\|\nabla^2 f(x^k)\|\le U_H$ and the fact that each iteration of Algorithm~\ref{pro:meo} requires only one Hessian-vector product of $f$, one can observe that the number of Hessian-vector products required by each call of Algorithm~\ref{pro:meo} is at most $\widetilde{\cO}(\min\{n,(U_H/\epsilon_H)^{1/2}\})$. Using these and statement (i) of Theorem~\ref{thm:c-pd-sosp}, we see that statement (ii) of Theorem~\ref{thm:c-pd-sosp} holds.
\end{proof}

\subsection{Proof of the main results in Section~\ref{sec:ncg}}\label{subsec:proof2}
In this subsection, we establish several technical lemmas, and then provide a proof of Theorems~\ref{thm:wd-ncg}, \ref{thm:c-ncg}, and \ref{thm:c-ncg-sosp}.


The following lemma presents some useful properties of the output of Algorithm~\ref{alg:capped-CG} when applied to solving the damped Newton system~\eqref{damp-n-trial}. It is a direct consequence of Lemma~\ref{lem:ppt-cg} given in Appendix~\ref{appendix:capped-CG}.

\begin{lemma}\label{lem:ppty-CG-1}
Suppose that Assumption~\ref{asp:NCG-cmplxity} holds and the direction $d_k^t$ results from Algorithm~\ref{alg:capped-CG} with a type specified in d$\_$type at some inner iteration $t$ of Algorithm~\ref{alg:NCG}. Then the following statements hold.
\begin{enumerate}[{\rm (i)}]
\item If d$\_$type=SOL, then $d_k^t$ satisfies 
\begin{eqnarray}
&(\sigma_t\epsilon_g)^{1/2}\|d_k^t\|^2\le (d_k^t)^T(\nabla^2 f(x^k)+2(\sigma_t\epsilon_g)^{1/2}I)d_k^t,\label{SOL-ppty-1-}\\
&\|d_k^t\|\le 1.1(\sigma_t\epsilon_g)^{-1/2}\|\nabla f(x^k)\|,\label{SOL-ppty-upbd-}\\
&(d_k^t)^T\nabla f(x^k)=-(d_k^t)^T(\nabla^2f(x^k)+2(\sigma_t\epsilon_g)^{1/2}I)d_k^t,\label{SOL-ppty-2-}\\
&\|(\nabla^2f(x^k)+2(\sigma_t\epsilon_g)^{1/2}I)d_k^t+\nabla f(x^k)\|\le \zeta(\sigma_t\epsilon_g)^{1/2}\|d_k^t\|/2.\label{SOL-ppty-3-}
\end{eqnarray}
\item If d$\_$type=NC, then $d^t_k$ satisfies $\nabla f(x^k)^Td_k^t\le 0$ and 
\begin{equation*}
(d_k^t)^T\nabla^2 f(x^k) d_k^t/\|d_k^t\|^2=- \min\{1,\sigma_t\}\|d_k^t\|\le- (\sigma_t\epsilon_g)^{1/2}.
\end{equation*}
\end{enumerate}
\end{lemma}

The following lemma generalizes Lemma~\ref{lem:next-FOSP-or-large-step} for Algorithm~\ref{alg:NCG} with any $\sigma_t\ge\gamma_\nu(\epsilon_g)$.


\begin{lemma}\label{lem:NCG-pid-next-FOSP}
Suppose that Assumption~\ref{asp:NCG-cmplxity} holds and the direction $d_k^t$ results from Algorithm \ref{alg:capped-CG} with d$\_$type=SOL at some inner iteration $t$ of the $k$th outer iteration of Algorithm~\ref{alg:NCG}. Assume that $\sigma_t\ge \gamma_\nu(\epsilon_g)$ holds, where $\gamma_\nu(\epsilon_g)$ is given in \eqref{gma-eps}. Then, either both $f(x^k+d_k^t)\le f(x^k)$ and $\|\nabla f(x^k + d_k^t)\|\le\epsilon_g$ hold, or $6\|d_k^t\|\ge (\epsilon_g/\sigma_t)^{1/2}$ holds.
\end{lemma}

\begin{proof}
Since $d_k^t$ results from Algorithm~\ref{alg:capped-CG} with d$\_$type=SOL, we see that $\|\nabla f(u)\|>\epsilon_g$ and \eqref{SOL-ppty-1-}-\eqref{SOL-ppty-3-} hold for $d_k^t$. Moreover, by by $\|\nabla f(x^k)\|>\epsilon_g$ and \eqref{SOL-ppty-3-}, we conclude that $d_k^t\neq 0$. To prove this lemma, it suffices to show that both $f(x^k+d_k^t)\le f(x^k)$ and $\|\nabla f(x^k + d_k^t)\|\le\epsilon_g$ hold under the condition $6\|d_k^t\|< (\epsilon_g/\sigma_t)^{1/2}$. To this end, we assume that $6\|d_k^t\|< (\epsilon_g/\sigma_t)^{1/2}$ holds throughout the remainder of this proof. 

We first prove $f(x^k+d_k^t)\le f(x^k)$. Suppose for contradiction that $f(x^k + d_k^t)> f(x^k)$. Using the same arguments as for \eqref{equi-alphastar-ge} with $(d^k,\gamma_\nu(\epsilon_g))$ replaced by $(d_k^t,\sigma_t)$, we can have $\|d^t_k\|^\nu \ge (1+\nu)(\sigma_t\epsilon_g)^{1/2}/H_\nu$. Letting $\gamma=\sigma_t$ in Lemma~\ref{lem:tech-gammamu-eps}, we observe from \eqref{rela:Hnu-gammanu-1} that $(\sigma_t\epsilon_g)^{1/2}/H_\nu \ge 2^{1+\nu} (\epsilon_g/\sigma_t)^{\nu/2}$. Using these, we obtain that $\|d_k^t\|^\nu\ge(1+\nu)2^{1+\nu} (\epsilon_g/\sigma_t)^{\nu/2}$, which contradicts the assumption that $6\|d_k^t\|< (\epsilon_g/\sigma_t)^{1/2}$, given that $\nu\in[0,1]$. Hence, $f(x^k+d_k^t)\le f(x^k)$ holds as desired.


We now prove $\|\nabla f(x^k + d_k^t)\|\le\epsilon_g$. Recall that $6\|d_k^t\|< (\epsilon_g/\sigma_t)^{1/2}$, $\sigma_t\ge\gamma_\nu(\epsilon_g)$, and $L(\epsilon_g/2)\le\gamma_\nu(\epsilon_g)/4$ (see \eqref{L1-bd-gma} with $a=2$). By these and the same arguments as for \eqref{gradxkdk-sol-1} with $(d^k,\gamma_\nu(\epsilon_g))$ replaced by $(d_k^t,\sigma_t)$, one can have
\begin{align*}
\|\nabla f(x^k + d_k^t)\|&\le \frac{L(\epsilon_g/2)}{2}\|d_k^t\|^2 + \frac{\epsilon_g}{2} + \frac{4+\zeta}{2}(\sigma_t\epsilon_g)^{1/2}\|d_k^t\|\\
&< \frac{L(\epsilon_g/2)}{72}\frac{\epsilon_g}{\sigma_t} + \frac{\epsilon_g}{2} + \frac{4+\zeta}{12}\epsilon_g \le \frac{\gamma_\nu(\epsilon_g)}{288}\frac{\epsilon_g}{\sigma_t} + \frac{\epsilon_g}{2} + \frac{4+\zeta}{12}\epsilon_g \le \epsilon_g,
\end{align*}
where the second inequality is due to $6\|d_k^t\|< (\epsilon_g/\sigma_t)^{1/2}$, the third inequality follows from $L(\epsilon_g/2)\le\gamma_\nu(\epsilon_g)/4$, and the last inequality is due to $\sigma_t\ge\gamma_\nu(\epsilon_g)$. Hence, $\|\nabla f(x^k + d_k^t)\|\le\epsilon_g$ holds.
\end{proof}

The next lemma shows that when $d^t_k$ generated in Algorithm~\ref{alg:NCG} is associated with d$\_$type=SOL and $\sigma_t\ge\gamma_\nu(\epsilon)$, Algorithm~\ref{alg:NCG} breaks its inner loop at the inner iteration $t$.


\begin{lemma}\label{lem:well-defined-pid-sol}
Suppose that Assumption~\ref{asp:NCG-cmplxity} holds and the direction $d_k^t$ results from Algorithm~\ref{alg:capped-CG} with d$\_$type=SOL at some inner iteration $t$ of the $k$th outer iteration of Algorithm~\ref{alg:NCG}. Assume that $\sigma_t\ge\gamma_\nu(\epsilon_g)$ holds, where $\gamma_\nu(\epsilon_g)$ is given in \eqref{gma-eps}. Then, either both $f(x^k+d_k^t)\le f(x^k)$ and $\|\nabla f(x^k+d_k^t)\|\le\epsilon_g$ hold, or there exists some nonnegative integer $j$ satisfying \eqref{sol-lwbd-thetaj} and \eqref{ls-sol-stepsize-2}.
\end{lemma}

\begin{proof}
Since $d_k^t$ results from Algorithm~\ref{alg:capped-CG} with d$\_$type=SOL, one can see that $\|\nabla f(x^k)\|>\epsilon_g$ and \eqref{SOL-ppty-1-}-\eqref{SOL-ppty-3-} hold for $d_k^t$. 
To prove this lemma, it suffices to show that if at least one of $f(x^k+d_k^t)\le f(x^k)$ and $\|\nabla f(x^k + d_k^t)\|\le\epsilon_g$ does not hold, there exists some nonnegative integer $j$ satisfying \eqref{sol-lwbd-thetaj} and \eqref{ls-sol-stepsize-2}.  To this end, we assume throughout the remainder of this proof that at least one of $f(x^k+d_k^t)\le f(x^k)$ and $\|\nabla f(x^k + d_k^t)\|\le\epsilon_g$ does not hold.  It then follows from Lemma~\ref{lem:NCG-pid-next-FOSP} that $6\|d_k^t\|\ge (\epsilon_g/\sigma_t)^{1/2}$. 

If \eqref{ls-sol-stepsize-2} holds with $j=0$,  then \eqref{sol-lwbd-thetaj}
and \eqref{ls-sol-stepsize-2} hold for $j=0$ and hence the conclusion of this lemma holds. Now, we suppose that \eqref{ls-sol-stepsize-2} is violated by some $j\ge0$. Using the same arguments as for \eqref{claim:lbd-stepsize} with $(d^k,\gamma_\nu(\epsilon_g))$ replaced by $(d_k^t,\sigma_t)$, we can have that all $j\ge0$ violating \eqref{ls-sol-stepsize-2} satisfy
\begin{align}\label{pf-thetaj-vio-lwbd}
\theta^{(1+\nu)j} \ge (1-\eta)(\sigma_t\epsilon_g)^{1/2}/(H_\nu \|d^t_k\|^\nu).
\end{align}
Recall that $\sigma_t\ge\gamma_\nu(\epsilon_g)$. Then, letting $\gamma=\sigma_t$ in Lemma~\ref{lem:tech-gammamu-eps}, we have from \eqref{rela:Hnu-gammanu-1} that $(\sigma_t\epsilon_g)^{1/2}/H_\nu \ge 2^{1+\nu} (\epsilon_g/\sigma_t)^{\nu/2}$, which together with \eqref{pf-thetaj-vio-lwbd} implies that all $j\ge0$ violating \eqref{ls-sol-stepsize-2} satisfy
\[
\theta^{(1+\nu)j} \ge (1-\eta) 2^{1+\nu} (\epsilon_g/\sigma_t)^{\nu/2}/\|d^t_k\|^\nu.
\]
Taking the $\left(\frac{1}{1+\nu}\right)$th root of both sides of the above inequality, and using $6\|d_k^t\|\ge (\epsilon_g/\sigma_t)^{1/2}$ and $\eta\in(0,1)$, we deduce that
\begin{align*}
\theta^j & \ge 2(1-\eta)^{1/(1+\nu)}\left(\frac{ (\epsilon_g/\sigma_t)^{1/4}}{\|d^t_k\|^{1/2}}\right)^{2\nu/(1+\nu)}\nonumber\\
& \ge 2 (1-\eta)\left(\frac{ (\epsilon_g/\sigma_t)^{1/4}}{\|d^t_k\|^{1/2}}\right)^{2\nu/(1+\nu)} \left(\frac{ (\epsilon_g/\sigma_t)^{1/4}}{\sqrt{6}\|d^t_k\|^{1/2}}\right)^{(1-\nu)/(1+\nu)} \ge \frac{2(1-\eta) (\epsilon_g/\sigma_t)^{1/4}}{3\|d^t_k\|^{1/2}}.
\end{align*}
By this and $\theta\in(0,1)$, one can observe that all $j\ge0$ that violate \eqref{ls-sol-stepsize-2} must be bounded above. Let $j_t$ be the smallest integer satisfying \eqref{ls-sol-stepsize-2}. Then, $j=j_t-1$ satisfies the above inequality and hence
\[
\theta^{j_t} \ge 2(1-\eta)\theta (\epsilon_g/\sigma_t)^{1/4}/(3\|d^t_k\|^{1/2}).
\]
It follows that such $j_t$ satisfies both \eqref{sol-lwbd-thetaj} and \eqref{ls-sol-stepsize-2}. Hence, the conclusion of this lemma holds.
\end{proof}

The following lemma shows that when $d^t_k$ generated in Algorithm~\ref{alg:NCG} is associated with d$\_$type=NC and $\sigma_t\ge\gamma_\nu(\epsilon)$, Algorithm~\ref{alg:NCG} breaks its inner loop at the inner iteration $t$.


\begin{lemma}\label{lem:well-defined-pid-nc}
Suppose that Assumption~\ref{asp:NCG-cmplxity} holds and the direction $d_k^t$ results from Algorithm \ref{alg:capped-CG} with d$\_$type=NC at some inner iteration $t$ of the $k$th outer iteration of Algorithm~\ref{alg:NCG}. Assume that $\sigma_t\ge\gamma_\nu(\epsilon_g)$ holds, where $\gamma_\nu(\epsilon_g)$ is given in \eqref{gma-eps}. Then, there exists some nonnegative integer $j$ satisfying \eqref{nc-lwbd-thetaj} and \eqref{ls-nc-stepsize-2}. 
\end{lemma}

\begin{proof}
Since $d^t_k$ is associated with d$\_$type of NC, we observe from Lemma~\ref{lem:ppty-CG-1}(ii) that 
\begin{equation*}
\nabla f(x^k)^Td^t_k\le 0,\quad (d^t_k)^T\nabla^2 f(x^k) d^t_k/\|d^t_k\|^2 = - \min\{1,\sigma_t\}\|d_k^t\|\le -(\sigma_t\epsilon_g)^{1/2}.
\end{equation*}	 
If \eqref{ls-nc-stepsize-2} holds with $j=0$, then \eqref{nc-lwbd-thetaj} and \eqref{ls-nc-stepsize-2} hold for $j=0$ and hence the conclusion of this lemma holds. Now, we suppose that  \eqref{ls-nc-stepsize-2} is violated by some $j\ge0$. Using the same arguments as for \eqref{lwbd-j-violate-ls-nc} with $(d^k,\gamma_\nu(\epsilon_g))$ replaced by $(d^t_k,\sigma_t)$, we can have that all $j\ge0$ violating \eqref{ls-nc-stepsize-2} satisfy
\begin{align}\label{lwbd-pf-nc-thetanuj}
\theta^{\nu j} > (1-\eta/2)\min\{1,\sigma_t\}^\nu (\sigma_t \epsilon_g)^{(1-\nu)/2}/H_\nu.
\end{align}
Recall that $\sigma_t\ge \gamma_\nu(\epsilon_g)$. Then, letting $\gamma=\sigma_t$ in Lemma~\ref{lem:tech-gammamu-eps}, and using \eqref{rela:Hnu-gammanu-2}, one has
$(\sigma_t \epsilon_g)^{(1-\nu)/2}/H_\nu \ge 2^{1+\nu}/\sigma_t^\nu$, which along with \eqref{lwbd-pf-nc-thetanuj} and $\eta\in(0,1)$ implies that all $j\ge0$ violating \eqref{ls-nc-stepsize-2} satisfy
\begin{equation}\label{lwbd-pf-nc-thetanuj-}
\theta^{\nu j} > (1-\eta/2) \min\{1,\sigma_t\}^\nu 2^{1+\nu}/\sigma_t^\nu > 2^\nu \min\{1,1/\sigma_t\}^\nu.    
\end{equation}
When $\nu=0$, one can see that \eqref{lwbd-pf-nc-thetanuj-} does not hold for $j=0$, which implies that \eqref{ls-nc-stepsize-2}  holds for $j=0$. Also,  \eqref{nc-lwbd-thetaj} holds for $j=0$ due to $\theta\in(0,1)$. When $\nu\in(0,1]$, one can see from \eqref{lwbd-pf-nc-thetanuj-} that  all $j\ge 0$ violating \eqref{ls-nc-stepsize-2} must be bounded above. Consequently, there exists some $j \geq 0$ such that \eqref{ls-nc-stepsize-2} holds. Let $j_t$ be the smallest nonnegative integer satisfying \eqref{ls-nc-stepsize-2}. Then $j_t=0$ or \eqref{lwbd-pf-nc-thetanuj-} holds for $j=j_t-1$. This together with $\theta\in(0,1)$ implies that $\theta^{j_t}\ge \theta\min\{1,1/\sigma_t\}$. Hence, both \eqref{nc-lwbd-thetaj} and \eqref{ls-nc-stepsize-2} hold for $j=j_t$. This completes the proof of this lemma.
\end{proof}

We are now ready to prove Theorem~\ref{thm:wd-ncg}.

\begin{proof}[Proof of Theorem~\ref{thm:wd-ncg}]

We first show that the number of calls of Algorithm~\ref{alg:capped-CG} at each outer iteration $k$ of Algorithm~\ref{alg:NCG} is at most $T$, where $T$ is given in \eqref{inner-bd}. To this end, let us consider an arbitrary outer iteration $k$ of Algorithm~\ref{alg:NCG}. 
 Clearly, this statement  holds if Algorithm~\ref{alg:capped-CG} is not invoked at the iteration $k$. Now, suppose that Algorithm~\ref{alg:capped-CG} is invoked at the iteration $k$. If Algorithm~\ref{alg:NCG} breaks its inner loop at $t=0$, then the number of calls of Algorithm~\ref{alg:capped-CG} is $1$, which is clearly bounded above by $T$. If Algorithm~\ref{alg:NCG} does not break its inner loop at $t=0$, one can see from Lemmas~\ref{lem:NCG-pid-next-FOSP}, \ref{lem:well-defined-pid-sol} and \ref{lem:well-defined-pid-nc} that Algorithm~\ref{alg:NCG} must break its inner loop at $t=t_k$ for some $t_k \geq 1$ and 
$\sigma_{t_k-1} <\gamma_\nu(\epsilon_g)$. Using, \eqref{inner-bd}, and the fact that $\sigma_{t_k}=r\sigma_{t_k-1}$, we have 
\beq \label{gamma_k}
\sigma_{t_k}=r\sigma_{t_k-1}<r\gamma_\nu(\epsilon_g)\le \sigma(\epsilon_g).
\eeq
In addition, notice from Algorithm~\ref{alg:NCG} that $\sigma_{t_k}= r^{t_k}\sigma_0\ge r^{t_k}\gamma_{-1}$.  It then follows that
$r^{t_k}\gamma_{-1}\le \sigma(\epsilon_g)$, which implies that $t_k \leq T-1$. Hence,  the number of calls of Algorithm~\ref{alg:capped-CG} at the outer iteration $k$ of Algorithm~\ref{alg:NCG} is at most $T$. 

We next show that $\gamma_k\le\sigma(\epsilon_g)$ for all $k\ge-1$ by induction. Indeed, $\gamma_{-1}\le \sigma(\epsilon_g)$ holds due to the definition of $\sigma(\epsilon_g)$. Now, suppose that $\gamma_{k-1}\le \sigma(\epsilon_g)$ holds for some $k\ge0$.  If $\|\nabla f(x^k)\|\le\epsilon_g$, we see from Algorithm~\ref{alg:NCG} that $\gamma_k=\gamma_{k-1}$ and hence $\gamma_k\le \sigma(\epsilon_g)$ holds. If $\|\nabla f(x^k)\|> \epsilon_g$ and Algorithm~\ref{alg:NCG} breaks its inner loop at $t=0$, then $\gamma_k=\sigma_0=\max\{\gamma_{-1},\gamma_{k-1}/r\}$, which, together with $r>1$ and the supposition $\gamma_{k-1}\le \sigma(\epsilon_g)$, implies that $\gamma_k \le\sigma(\epsilon_g)$.  Otherwise, Algorithm~\ref{alg:NCG} must break its inner loop at $t=t_k$ for some $t_k \geq 1$. By this and \eqref{gamma_k}, one has $\gamma_k=\sigma_{t_k}\le \sigma(\epsilon_g)$. 
 This completes the induction. Hence, $\gamma_k\le\sigma(\epsilon_g)$ holds as desired.

We finally show that the total number of calls of Algorithms~\ref{alg:capped-CG} and \ref{pro:meo} during the first $s$ outer iterations of Algorithm~\ref{alg:NCG} is at most $T+2s$. For convenience, we let $\tau_k$ denote the number of calls of Algorithms~\ref{alg:capped-CG} and \ref{pro:meo} in the outer iteration $k$ of Algorithm~\ref{alg:NCG}. If $\|\nabla f(x^k)\|\le\epsilon_g$, then Algorithm~\ref{pro:meo} is invoked in the outer iteration $k$, $\gamma_k=\gamma_{k-1}$, and $\tau_k=1$. Otherwise, Algorithm~\ref{alg:capped-CG} is invoked in the outer iteration $k$, and we have either ($\gamma_k =  \sigma_0$ and $\tau_k=1$) or 
($\gamma_k  = r^{\tau_k-1} \sigma_0$ and $\tau_k>1$). By these, $r>1$ and  $\sigma_0 = \max\{\gamma_{-1},\gamma_{k-1}/r\}$ (see Algorithm~\ref{alg:NCG}), one can obtain that $\gamma_k\ge r^{\tau_k-2}\gamma_{k-1}$ for all $k\ge0$. It then follows that
\[
\sum_{k=0}^{s-1}\tau_k \le \ln(\gamma_{s-1}/\gamma_{-1})/\ln r + 2s.
\]
where together with $\gamma_{s-1}\le \sigma(\epsilon_g)$ implies that the total number of calls of Algorithms~\ref{alg:capped-CG} and \ref{pro:meo} during the first $s$ outer iterations of Algorithm~\ref{alg:NCG} is at most $T+2s$. 
\end{proof}

The next lemma shows that when the search direction $d^k$ in Algorithm~\ref{alg:NCG} is of type `SOL', the next iterate $x^{k+1}$ either satisfies $\|\nabla f(x^{k+1})\|\le\epsilon_g$ or produces a sufficient decrease in $f$.

\begin{lemma}\label{suf-dec-sol1}
Suppose that Assumption~\ref{asp:NCG-cmplxity} holds and the direction $d^k$ results from Algorithm~\ref{alg:capped-CG} with d$\_$type=SOL at some outer iteration $k$ of Algorithm~\ref{alg:NCG}. Let $\hat{c}_{\mathrm{sol}}$ be defined in \eqref{hat-csol-cnc}. Then, the next iterate $x^{k+1}=x^k+\alpha_k d^k$ satisfies either $\|\nabla f(x^{k+1})\|\le\epsilon_g$ or 
\begin{equation}\label{sd-sol-ncg}
f(x^k) - f(x^{k+1}) \ge \hat{c}_{\mathrm{sol}}\epsilon_g^{3/2}/\gamma_k^{1/2}.
\end{equation}
\end{lemma}

\begin{proof}
Since $d^k$ results from Algorithm~\ref{alg:capped-CG} with d$\_$type=SOL, we observe from Algorithm~\ref{alg:NCG} and Lemma~\ref{lem:ppty-CG-1}(i) that $\|\nabla f(x^k)\|>\epsilon_g$  and  \eqref{SOL-ppty-1-}-\eqref{SOL-ppty-3-} hold with 
$(d_k^t,\sigma_t)$ replaced by $(d^k,\gamma_k)$. In addition, since the next iterate $x^{k+1}$ has already been generated, one can see from Algorithm~\ref{alg:capped-CG} that at least one of $\|\nabla f(x^{k+1})\|\le\epsilon_g$ and $6\|d^k\|\ge[\epsilon_g/\gamma_k]^{1/2}$ holds. Therefore, to prove this lemma, it suffices to show that  \eqref{sd-sol-ncg} holds if $\|\nabla f(x^{k+1})\|>\epsilon_g$. To this end, we suppose that $\|\nabla f(x^{k+1})\|>\epsilon_g$ holds through the remainder of the proof, which implies that $6\|d^k\|\ge[\epsilon_g/\gamma_k]^{1/2}$ holds. We next prove \eqref{sd-sol-ncg} by considering two separate cases below.

Case 1) $\alpha_k=1$. One can see from \eqref{ls-sol-stepsize-2} with $(d^t_k,\sigma_t,\theta^j)=(d^k,\gamma_k,1)$ that 
\begin{align*}
f(x^k + d^k) \le f(x^k) - \eta (\gamma_k\epsilon_g)^{1/2}\|d^k\|^2.
\end{align*}
Using this and $6\|d^k\|\ge[\epsilon_g/\gamma_k]^{1/2}$, we obtain that $f(x^k) - f(x^{k+1}) \ge \eta\epsilon_g^{3/2}/(36\gamma_k^{1/2})$, which together with the definition of $\hat{c}_{\mathrm{sol}}$ in \eqref{hat-csol-cnc} implies that \eqref{sd-sol-ncg} holds.

Case 2) $\alpha_k<1$. We can observe from Algorithm~\ref{alg:NCG} that $\alpha_k=\theta^{j_t}$, where $j_t$ satisfies \eqref{sol-lwbd-thetaj} and \eqref{ls-sol-stepsize-2} with $(d^t_k,\sigma_t)=(d^k,\gamma_k)$.  It then follows that
\begin{align*}
\alpha_k\ge 2(1-\eta)\theta(\epsilon_g/\gamma_k)^{1/4}/(3\|d^k\|^{1/2}),\quad  f(x^k+\alpha_k d^k) \le f(x^k) - \eta(\gamma_k \epsilon_g)^{1/2}\alpha_k^2\|d^k\|^2.   
\end{align*}
Using these inequalities, $x^{k+1}=x^k+\alpha_k d^k$, and $6\|d^k\|\ge [\epsilon_g/\gamma_k]^{1/2}$, we obtain that 
\[
f(x^k) - f(x^{k+1}) \ge \eta\left(\frac{2(1-\eta)\theta}{3}\right)^2\epsilon_g\|d^k\| \ge \frac{\eta}{6}\left(\frac{2(1-\eta)\theta}{3}\right)^2 \epsilon_g^{3/2}/\gamma_k^{1/2},
\]
which along with the definition of $\hat{c}_{\mathrm{sol}}$ in \eqref{hat-csol-cnc} implies that \eqref{sd-sol-ncg} holds.
\end{proof}

Our next lemma shows that when the search direction $d^k$ in Algorithm~\ref{alg:NCG} is of type `NC', the next iterate $x^{k+1}$ produces a sufficient decrease in $f$.

\begin{lemma}\label{suf-dec-nc1}
Suppose that Assumption~\ref{asp:NCG-cmplxity} holds and the direction $d^k$ results from Algorithm~\ref{alg:capped-CG} with d$\_$type=NC at some outer iteration $k$ of Algorithm~\ref{alg:NCG}. Let ${c}_{\mathrm{nc}}$ be defined as in \eqref{csol-cnc}. Then, the next iterate $x^{k+1} = x^k + \alpha_k d^k$ satisfies
\begin{equation}\label{sd-nc-ncg}
f(x^k) - f(x^{k+1}) \ge {c}_{\mathrm{nc}} \epsilon_g^{3/2}/\gamma_k^{1/2}.
\end{equation}
\end{lemma}

\begin{proof}
Since $d^k$ is associated with d$\_$type=NC, we observe from Lemma~\ref{lem:ppty-CG-1}(ii) with $(d^t_k,\sigma_t)=(d^k,\gamma_k)$ that 
\begin{align}\label{lwbd-nc-dk-gammak}
\|d^k\| \ge \max\{1,1/\gamma_k\}(\gamma_k \epsilon_g)^{1/2}.    
\end{align}
In addition, we observe from Algorithm~\ref{alg:NCG} that $\alpha_k=\theta^{j_t}$, where $j_t$ satisfies \eqref{nc-lwbd-thetaj} and \eqref{ls-nc-stepsize-2} with $(d^t_k,\sigma_t)=(d^k,\gamma_k)$. By these and \eqref{lwbd-nc-dk-gammak}, one has 
\begin{align*}
f(x^k) - f(x^{k+1}) & > \frac{\eta}{4}\min\{1,\gamma_k\}\alpha_k^2\|d^k\|^2 \\
&\ge \frac{\eta}{4} \min\{1,\gamma_k\} (\theta\min\{1,1/\gamma_k\})^2 (\max\{1,1/\gamma_k\})^3(\gamma_k\epsilon_g)^{3/2} = \frac{\eta\theta^2}{4}\epsilon_g^{3/2}/\gamma_k^{1/2},
\end{align*}
which along with the definition of ${c}_{\mathrm{nc}}$ in \eqref{csol-cnc} implies that \eqref{sd-nc-ncg} holds.
\end{proof}

We are now ready to prove Theorem~\ref{thm:c-ncg}.

\begin{proof}[Proof of Theorem~\ref{thm:c-ncg}]
(i) We first show that the number of outer iterations of Algorithm~\ref{alg:NCG} is at most $\overline{K}_1$. Suppose for contradiction that the number of its outer iterations is more than $\overline{K}_1$. By Lemmas~\ref{suf-dec-sol1} and \ref{suf-dec-nc1} and the assumption that  $\epsilon_H\in(0,1)$ is not provided, we observe that in each outer iteration of Algorithm~\ref{alg:NCG}, except for the last one, the function $f$ is reduced at least by $\min\{\hat{c}_{\mathrm{sol}},{c}_{\mathrm{nc}}\}\epsilon_g^{3/2}/\gamma_k^{1/2}$. Using this, \eqref{lwbd-Hgupbd},  and $\gamma_k\le\sigma(\epsilon_g)$ (see Theorem~\ref{thm:wd-ncg}), we have
\begin{align*}
\overline{K}_1 \min\{\hat{c}_{\mathrm{sol}},{c}_{\mathrm{nc}}\}\epsilon_g^{3/2}/\sigma(\epsilon_g)^{1/2}&\le \overline{K}_1 \min\{\hat{c}_{\mathrm{sol}},{c}_{\mathrm{nc}}\}\epsilon_g^{3/2}/\gamma_k^{1/2}\le\sum_{k=0}^{\overline{K}_1-1} [f(x^k) - f(x^{k+1})] \\
&=f(x^0) - f(x^{\overline{K}_1}) \le f(x^0) - \fl,
\end{align*} 
which contradicts the definition of $\overline{K}_1$ in \eqref{hat-K1}.   Hence, the number of outer iterations of Algorithm~\ref{alg:NCG} is at most $\overline{K}_1$. 

Recall from above that the number of outer iterations of Algorithm~\ref{alg:NCG} is at most $\overline{K}_1$.  Using this and  Theorem~\ref{thm:wd-ncg}, we see that the total number of calls of Algorithm~\ref{alg:capped-CG} in Algorithm~\ref{alg:NCG} is at most $T + 2\overline{K}_1$. This along with the fact that Algorithm~\ref{alg:capped-CG} is called once at each inner iteration of Algorithm~\ref{alg:NCG} implies that the total number of inner iterations of Algorithm~\ref{alg:NCG} is at most $T + 2\overline{K}_1$. In addition, relations~\eqref{order-pf-KP1} and \eqref{order-pf-TKP1} follow from \eqref{gma-eps}, \eqref{csol-cnc}, \eqref{inner-bd}, \eqref{hat-csol-cnc}, and  \eqref{hat-K1}. Since $\epsilon_H$ is not provided, one can observe that the output $x^k$ of Algorithm~\ref{alg:NCG} satisfies $\|\nabla f(x^k)\|\le\epsilon_g$. This completes the proof of statement (i) of Theorem~\ref{thm:c-ncg}.

(ii)  Recall from above that the number of outer iterations of Algorithm~\ref{alg:NCG} is at most $\overline{K}_1$. Suppose that Algorithm~\ref{alg:NCG} terminates at some outer iteration $K^\prime$ with $K^\prime< \overline{K}_1$. Notice from Lemmas~\ref{suf-dec-sol1} and \ref{suf-dec-nc1} that each outer iteration of Algorithm~\ref{alg:NCG}, except for the last one, results in a reduction on $f$ at least by $\min\{\hat{c}_{\mathrm{sol}},{c}_{\mathrm{nc}}\}\epsilon_g^{3/2}/\gamma_k^{1/2}$. Hence,
\[
\sum_{k=0}^{K^\prime-1}\min\{\hat{c}_{\mathrm{sol}},{c}_{\mathrm{nc}}\}\epsilon_g^{3/2}/\gamma_k^{1/2}\le\sum_{k=0}^{K^\prime-1}(f(x^k)-f(x^{k+1}))=f(x^0) - f(x^{K^\prime})\le f(x^0) -\fl,
\]
where the last inequality folllows from \eqref{lwbd-Hgupbd}. 
Rearranging the terms of this inequality, we obtain that
\beq\label{upbd-1/4}
\sum_{k=0}^{K^\prime-1}1/\gamma_k^{1/2} \le (f(x^0) -\fl) \epsilon_g^{-3/2}/\min\{\hat{c}_{\mathrm{sol}},{c}_{\mathrm{nc}}\}.
\eeq
In addition, notice that $f$ is descent along the iterates generated by Algorithm~\ref{alg:NCG}, which implies $f(x^k) \le f(x^0)$ for all $0 \leq k < K^\prime$. It then follows from \eqref{lwbd-Hgupbd} that $\|\nabla^2 f(x^k)\|\le U_H$ for all $0 \leq k < K^\prime$. 
By Theorem~\ref{lem:capped-CG} with $(H,\varepsilon)=(\nabla^2 f(x^k),(\sigma_t\epsilon_g)^{1/2})$ and $\|\nabla^2 f(x^k)\|\le U_H$, we obtain that in the $k$th outer iteration of Algorithm~\ref{alg:NCG}, the number of gradient evaluations and Hessian-vector products required by the call of Algorithm~\ref{alg:capped-CG} at its $t$th inner iteration is at most
\[
\min\Big\{n,\big\lceil(U_H^{1/2}/(\sigma_t\epsilon_g)^{1/4}+2)\psi(U_H/(\sigma_t\epsilon_g)^{1/2})\big\rceil\Big\},
\]
where $\psi$ is given in Theorem~\ref{lem:capped-CG}. Further, by $\sigma_t\ge \sigma_0=\max\{\gamma_{-1},\gamma_{k-1}/r\}$ and the monotonicity of $\psi$, one can see that the above quantity is bounded above by 
\[
\min\Big\{n,\big\lceil(U_H^{1/2}/(\gamma_{k-1}\epsilon_g/r)^{1/4}+2)\psi(U_H/(\gamma_{-1}\epsilon_g)^{1/2})\big\rceil\Big\}.
\]
Let $\tau_k$ denote  the number of calls of Algorithm~\ref{alg:capped-CG} (i.e., the number of inner iterations) in the $k$th outer iteration of Algorithm~\ref{alg:NCG}. It follows from statement (i) of Theorem~\ref{thm:c-ncg} that $\sum_{k=0}^{K^\prime-1} \tau_k\leq T + 2\overline{K}_1$, where $\overline{K}_1$ is given in \eqref{hat-K1}. Also, recall from Theorem~\ref{thm:wd-ncg} that $\tau_k\le T$ for all $0\le k\le K^\prime-1$. Based on these observations, we obtain that the total number of gradient evaluations and Hessian-vector products of $f$ required by the calls of Algorithm~\ref{alg:capped-CG} in Algorithm~\ref{alg:NCG} is bounded by
\begin{align*}
&\sum_{k=0}^{K^\prime-1}\tau_k \min\Big\{n,\big\lceil(U_H^{1/2}/(\gamma_{k-1}\epsilon_g/r)^{1/4}+2)\psi(U_H/(\gamma_{-1}\epsilon_g)^{1/2})\big\rceil\Big\} \nn\\
&\le \min\Big\{n \sum_{k=0}^{K^\prime-1}\tau_k,\sum_{k=0}^{K^\prime-1} \big([U_H^{1/2}/(\gamma_{k-1}\epsilon_g/r)^{1/4}+2]\psi(U_H/(\gamma_{-1}\epsilon_g)^{1/2})+1\big)\tau_k\Big\}\nn\\
&= \min\Big\{n \sum_{k=0}^{K^\prime-1}\tau_k,U_H^{1/2}(r/\epsilon_g)^{1/4}\psi(U_H/(\gamma_{-1}\epsilon_g)^{1/2})\sum_{k=0}^{K^\prime-1} \gamma_{k-1}^{-1/4}\tau_k + 2\big(\psi(U_H/(\gamma_{-1}\epsilon_g)^{1/2})+1\big)\sum_{k=0}^{K^\prime-1} \tau_k\Big\}\nn\\
&\leq \min\Bigg\{n \sum_{k=0}^{K^\prime-1}\tau_k,U_H^{1/2}(r/\epsilon_g)^{1/4}\psi(U_H/(\gamma_{-1}\epsilon_g)^{1/2})
T\sqrt{\overline{K}_1\sum_{k=0}^{K^\prime-1} \gamma_{k-1}^{-1/2}} + 2\big(\psi(U_H/(\gamma_{-1}\epsilon_g)^{1/2})+1\big)\sum_{k=0}^{K^\prime-1} \tau_k\Bigg\}\nn\\
&\leq \min\Big\{n (T + 2\overline{K}_1),U_H^{1/2}(r/\epsilon_g)^{1/4}\psi(U_H/(\gamma_{-1}\epsilon_g)^{1/2})T
\sqrt{\overline{K}_1\big((f(x^0) -\fl) \epsilon_g^{-3/2}/\min\{\hat{c}_{\mathrm{sol}},{c}_{\mathrm{nc}}\}+\gamma_{-1}^{-1/2}\big)}\nn \\ 
&\qquad \quad+ 2\big(\psi(U_H/(\gamma_{-1}\epsilon_g)^{1/2})+1\big)(T + 2\overline{K}_1)\Big\},\nn \\
&=\widetilde{\cO}\Big(\min\big\{n, U_H^{1/2}/(H_\nu^{1/(2+2\nu)}\epsilon_g^{-\nu/(2+2\nu)})\big\}H_\nu^{1/(1+\nu)}\epsilon_g^{-(2+\nu)/(1+\nu)}\Big),
\end{align*}
where the first inequality is due to $\min\{a_1,b_1\}+\min\{a_2,b_2\}\le\min\{a_1+a_2,b_1+b_2\}$ for all $a_1,a_2,b_1,b_2\in\bR$, the second inequality follows from Cauchy-Schwarz inequality and $(\sum_{k=0}^{K^\prime-1} \tau_k^2)^{1/2} \leq T({K}^\prime)^{1/2} < T \overline{K}_1^{1/2}$ because $\tau_k\le T$ for all $0\le k\le K^\prime-1$ and $K^\prime<\overline{K}_1$,
the last inequality is due to $\sum_{k=0}^{K^\prime-1} \tau_k\leq T + 2\overline{K}_1$ and \eqref{upbd-1/4}, and the last equality follows from \eqref{order-pf-KP1} and \eqref{order-pf-TKP1} and the definition of $\psi$ in Theorem \ref{lem:capped-CG}.  Hence, statement (ii) of Theorem~\ref{thm:c-ncg} holds.
\end{proof}

The following lemma shows that when the search direction $d^k$ in Algorithm~\ref{alg:NCG} is a negative curvature direction returned from Algorithm~\ref{pro:meo}, the next iterate $x^{k+1}$ produces a sufficient reduction on $f$. Its proof is identical to that of Lemma~\ref{lem:meo-dec}, and thus is omitted here.

\begin{lemma}\label{lem:meo-dec1}
Suppose that Assumption~\ref{asp:NCG-cmplxity} holds with $H_\nu>0$ and $\nu\in(0,1]$, and $d^k$ results from Algorithm~\ref{pro:meo} at some outer iteration $k$ of Algorithm~\ref{alg:NCG}. Let $c_{\mathrm{meo}}$ be defined in \eqref{cmeo}. Then the following statements hold.
\begin{enumerate}[{\rm (i)}]
\item The step length $\alpha_k$ is well defined, and $\alpha_k\ge\min\big\{1,\theta((1-\eta)/H_\nu)^{1/\nu}(\epsilon_H/2)^{(1-\nu)/\nu}\big\}$.
\item The next iterate $x^{k+1}=x^k+\alpha_k d^k$ satisfies $f(x^k) - f(x^{k+1}) \ge c_{\mathrm{meo}} \epsilon_H^{(2+\nu)/\nu}$.
\end{enumerate}
\end{lemma}

We are now ready to prove Theorem~\ref{thm:c-ncg-sosp}.

\begin{proof}[Proof of Theorem~\ref{thm:c-ncg-sosp}]
(i) Let $K_2$ and $\overline{K}_1$  be defined in \eqref{K2} and \eqref{hat-K1}, respectively.  Observe from  Algorithm~\ref{alg:NCG} and Lemma~\ref{lem:meo-dec1}(ii) that each call of Algorithm~\ref{pro:meo}, except the last one, results in a reduction on $f$ at least by $c_{\mathrm{meo}}\epsilon_H^{(2+\nu)/\nu}$. By this and similar arguments as used in the proof of Theorem \ref{thm:c-pd-sosp}(i), one can claim that the total number of calls of Algorithm~\ref{pro:meo} in Algorithm~\ref{alg:NCG} is at most $K_2$. 

In addition, we claim that the total number of calls of Algorithm~\ref{alg:capped-CG} in Algorithm~\ref{alg:NCG} is at most $\overline{K}_1+K_2-1$. Indeed, suppose for contradiction that its total number of calls is more than $\overline{K}_1+K_2-1$. Notice that if Algorithm~\ref{alg:capped-CG} is called at some iteration $k$ and generates $x^{k+1}$ satisfying $\|\nabla f(x^{k+1})\| \le \epsilon_g$, then Algorithm~\ref{pro:meo} must be called at the iteration $k+1$. In view of this and the fact that the total number of calls of Algorithm~\ref{pro:meo} is at most $K_2$, one can observe that the total number of such iterations is at most $K_2$. This along with the above supposition implies that the total number of iterations $k$ of Algorithm~\ref{alg:NCG} at which Algorithm~\ref{alg:capped-CG} is called and generates the next iterate $x^{k+1}$ satisfying $\|\nabla f(x^{k+1})\|>\epsilon_g$ is at least $\overline{K}_1$. For each of these iterations $k$, we observe from Lemmas~\ref{lem:sol-1}(ii) and \ref{lem:nc-1}(ii) and Theorem~\ref{thm:wd-ncg} that 
\[
f(x^k) - f(x^{k+1}) \ge \min\{\hat{c}_{\mathrm{sol}},{c}_{\mathrm{nc}}\}\epsilon_g^{3/2}/\gamma_k^{1/2} \geq  \min\{\hat{c}_{\mathrm{sol}},{c}_{\mathrm{nc}}\}\epsilon_g^{3/2}/\sigma(\epsilon_g)^{1/2}.
\]  
Since $f$ is descent along the iterates of Algorithm~\ref{alg:NCG} and $f(x^k) \geq \fl$,  the total amount of reduction on $f$ resulting from these iterations $k$ is at most $f(x^0) - \fl$. It then follows that
\begin{align*}
\overline{K}_1 \min\{c_{\mathrm{sol}},c_{\mathrm{nc}}\}\epsilon_g^{3/2}/\gamma_\nu(\epsilon_g)^{1/2} \le f(x^0) - \fl,     
\end{align*}
which contradicts the definition of $\overline{K}_1$ given in \eqref{hat-K1}. Hence, the total number of calls of Algorithm~\ref{alg:capped-CG} in Algorithm~\ref{alg:NCG} is at most $\overline{K}_1+K_2-1$. 

Based on the above claims and the fact that either Algorithm~\ref{alg:capped-CG} or \ref{pro:meo} is called at each iteration of Algorithm~\ref{alg:NCG}, we conclude that the total number of iterations of Algorithm~\ref{alg:NCG} is at most $\overline{K}_1+2K_2-1$.  Using this and  Theorem~\ref{thm:wd-ncg}, we see that the total number of calls of Algorithm~\ref{alg:capped-CG} in Algorithm~\ref{alg:NCG} is at most $T + 2\overline{K}_1+ 4K_2 - 2$. This along with the fact that Algorithm~\ref{alg:capped-CG} is called once at each inner iteration of Algorithm~\ref{alg:NCG} implies that the total number of inner iterations of Algorithm~\ref{alg:NCG} is at most $T + 2\overline{K}_1+ 4K_2 - 2$. In addition, relations~\eqref{outer-iter-ncg-pf} and \eqref{inner-iter-ncg-pf} follow from \eqref{gma-eps}, \eqref{csol-cnc},  \eqref{cmeo}, \eqref{K2}, \eqref{inner-bd},  \eqref{hat-csol-cnc}, and \eqref{hat-K1}. Moreover, one can observe that the output $x^k$ of Algorithm~\ref{alg:NCG} satisfies $\|\nabla f(x^k)\|\le\epsilon_g$ deterministically and $\lambda_{\min}(\nabla^2 f(x^k))\ge-\epsilon_H$ with probability at least $1-\delta$ for some $0\le k\le \overline{K}_1+ 2K_2 -1$, where the latter part is due to Algorithm~\ref{pro:meo}. This completes the proof of statement (i) of Theorem~\ref{thm:c-ncg-sosp}.

(ii) Notice that $f$ is descent along the iterates generated by Algorithm~\ref{alg:NCG}, which implies $f(x^k) \le f(x^0)$ for each iteration $k$. It then follows from \eqref{lwbd-Hgupbd} that $\|\nabla^2 f(x^k)\|\le U_H$ for each iteration $k$. Recall from the above proof that the total number of iterations of Algorithm~\ref{alg:NCG} is at most $\overline{K}_1+2K_2-1$ and the total number of calls of Algorithm~\ref{alg:capped-CG} in Algorithm~\ref{alg:NCG} is at most $T + 2\overline{K}_1+ 4K_2 - 2$. In view of Theorem~\ref{lem:capped-CG} with $(H,\varepsilon)=(\nabla^2 f(x^k),(\sigma_t\epsilon_g)^{1/2})$ and the fact that $\|\nabla^2 f(x^k)\|\le U_H$ and $\sigma_t\ge\gamma_{-1}$, one can observe that the number of gradient evaluations and Hessian-vector products of $f$ required by each call of Algorithm~\ref{alg:capped-CG} with input $U=0$ is at most $\widetilde{\cO}(\min\{n,U_H^{1/2}/\epsilon_g^{1/4}\})$. Using these, we obtain that the total number of gradient evaluations and Hessian-vector products of $f$ required by the calls of Algorithm~\ref{alg:capped-CG} in Algorithm~\ref{alg:NCG} is bounded by
\begin{align}
\widetilde{\cO}((T + \overline{K}_1+ K_2)\min\{n,U_H^{1/2}/\epsilon_g^{1/4}\}). \label{upbd-oper}
\end{align}
In addition, by Theorem~\ref{rand-Lanczos} with $(H,\varepsilon)=(\nabla^2 f(x^k),\epsilon_H)$, $\|\nabla^2 f(x^k)\|\le U_H$ and the fact that each iteration of the Lanczos method requires only one Hessian-vector product of $f$, one can observe that the number of Hessian-vector products required by each call of Algorithm~\ref{pro:meo} in Algorithm~\ref{alg:NCG} is at most $\widetilde{\cO}(\min\{n,(U_H/\epsilon_H)^{1/2}\})$. Recall from the above proof that the total number of calls of Algorithm~\ref{pro:meo} in Algorithm~\ref{alg:NCG} is at most $K_2$. Hence, the total number of Hessian-vector products required by all calls of Algorithm~\ref{pro:meo} in Algorithm~\ref{alg:NCG} is at most $\widetilde{\cO}(K_2\min\{n,(U_H/\epsilon_H)^{1/2}\})$. Using this, \eqref{inner-iter-ncg-pf}, and \eqref{upbd-oper}, 
we see that statement (ii) of Theorem~\ref{thm:c-ncg-sosp} holds.
\end{proof}



\subsubsection*{Acknowledgments}
The first author was partially supported by the Wallenberg AI, Autonomous Systems and Software Program (WASP) funded by the Knut and Alice Wallenberg Foundation. The second author was partially supported by National Science Foundation Award IIS-2347592. The third author was partially supported by the National Science Foundation Award IIS-2211491, the Office of Naval Research Award N00014-24-1-2702, and the Air Force Office of Scientific Research Award FA9550-24-1-0343.

\bibliographystyle{abbrv}
\bibliography{references}

\section*{Appendix}

\appendix

\section{A capped conjugate gradient method}\label{appendix:capped-CG}

We present a capped CG method in Algorithm \ref{alg:capped-CG}, which was  proposed in \cite[Algorithm~1]{RNW18} for finding either an approximate solution to the linear system \eqref{indef-sys} or a sufficiently negative curvature direction of the associated coefficient matrix. Its details can be found in \cite[Section~3.1]{RNW18}.

\begin{algorithm}[h]
\caption{A capped conjugate gradient method}
\label{alg:capped-CG}
{\footnotesize
\begin{algorithmic}
\State \noindent\textit{Inputs}: symmetric matrix $H\in\bR^{n\times n}$, vector $g\neq0$, damping parameter $\varepsilon>0$, desired relative accuracy $\zeta\in(0,1)$.
\State \textit{Optional input:} scalar $U\ge0$ (set to $0$ if not provided).
\State \textit{Outputs:} d$\_$type, $d$.
\State \textit{Secondary outputs:} final values of $U,\kappa,\hat{\zeta},\tau,$ and $T$.
\State Set
\begin{equation*}
\bar{H}:=H+2\varepsilon I,\quad \kappa:=\frac{U+2\varepsilon}{\varepsilon},\quad\hat{\zeta}:=\frac{\zeta}{3\kappa},\quad\tau:=\frac{\sqrt{\kappa}}{\sqrt{\kappa}+1},\quad T:=\frac{4\kappa^4}{(1-\sqrt{\tau})^2},
\end{equation*}
$y^0\leftarrow 0,r^0\leftarrow g,p^0\leftarrow -g, j\leftarrow 0$.
\If {$(p^0)^T \bar{H}p^0<\varepsilon\|p^0\|^2$}
\State Set $d\leftarrow p^0$ and terminate with d$\_$type = NC;
			\ElsIf {\ $\|Hp^0\|>U\|p^0\|\ $}
			\State Set $U\leftarrow\|Hp^0\|/\|p^0\|$ and update $\kappa,\hat{\zeta},\tau, T$ accordingly;
			\EndIf
			\While{TRUE}
			\State $\alpha_j\leftarrow (r^j)^T r^j/(p^j)^T\bar{H}p^j$; \{Begin Standard CG Operations\}
			\State $y^{j+1}\leftarrow y^j+\alpha_jp^j$;
			\State $r^{j+1}\leftarrow r^j+\alpha_j\bar{H}p^j$;
			\State $\beta_{j+1}\leftarrow\|r^{j+1}\|^2/\|r^j\|^2$;
			\State $p^{j+1}\leftarrow-r^{j+1}+\beta_{j+1}p^j$; \{End Standard CG Operations\}
			\State $j\leftarrow j+1$;
			\If {$\|Hp^j\|>U\|p^j\|$}
			\State Set $U\leftarrow\|Hp^j\|/\|p^j\|$ and update $\kappa,\hat{\zeta},\tau,T$ accordingly;
			\EndIf
			\If {\ $\|Hy^j\|>U\|y^j\|\ $}
			\State Set $U\leftarrow\|Hy^j\|/\|y^j\|$ and update $\kappa,\hat{\zeta},\tau,T$ accordingly;
			\EndIf
			\If {\ $\|Hr^j\|>U\|r^j\|\ $}
			\State Set $U\leftarrow\|Hr^j\|/\|r^j\|$ and update $\kappa,\hat{\zeta},\tau,T$ accordingly;
			\EndIf
			\If {$(y^j)^T\bar{H}y^j<\varepsilon\|y^j\|^2$}
			\State Set $d\leftarrow y^j$ and terminate with d$\_$type = NC;
			\ElsIf {\ $\|r^j\|\le\hat{\zeta}\|r^0\|$}
			\State Set $d\leftarrow y^j$ and terminate with d$\_$type = SOL;
			\ElsIf{\ $(p^j)^T\bar{H}p^j<\varepsilon\|p^j\|^2$}
			\State Set $d\leftarrow p^j$ and terminate with d$\_$type = NC;
			\ElsIf {\ $\|r^j\|>\sqrt{T}\tau^{j/2}\|r^0\| $}
			\State Compute $\alpha_j, y^{j+1}$ as in the main loop above;
			\State Find $i\in\{0,\ldots,j-1\}$ such that
			\[
			(y^{j+1}-y^i)^T\bar{H}(y^{j+1}-y^i)<\varepsilon\|y^{j+1}-y^i\|^2;
			\]
			\State Set $d\leftarrow y^{j+1}-y^i$ and terminate with d$\_$type = NC;
			\EndIf
			\EndWhile
		\end{algorithmic}
	}
\end{algorithm}

The following lemma present some useful properties of Algorithm~\ref{alg:capped-CG} below, which are adopted from \cite[Lemma 3]{RNW18}.

\begin{lemma}\label{lem:ppt-cg}
Consider applying Algorithm~\ref{alg:capped-CG} with input $U = 0$ to the linear system~\eqref{indef-sys} with $g\neq 0$, $\varepsilon>0$, and $H$ being an $n\times n$ symmetric matrix. Let $d$ be the output of Algorithm \ref{alg:capped-CG} with a type specified in d$\_$type. Then the following statements hold.
\begin{enumerate}[{\rm (i)}]
\item If d$\_$type=SOL, then  $d$ satisfies
\begin{align*}
&\varepsilon\|d\|^2\le d^T(H+2\varepsilon I)d,\qquad \|d\|\le 1.1\varepsilon^{-1}\| g\|,\\
&d^T g=-d^T(H+2 \varepsilon I)d,\qquad \|(H+2\varepsilon I)d+ g\|\le \zeta \varepsilon\|d\|/2.
\end{align*}
\item If d$\_$type=NC, then $d$ satisfies $d^T g\le0$ and $d^T H d/\|d\|^2\le-\varepsilon$.
\end{enumerate}
\end{lemma}

The following theorem presents the iteration complexity of Algorithm~\ref{alg:capped-CG}.

\begin{theorem}[{\bf iteration complexity of Algorithm \ref{alg:capped-CG}}]\label{lem:capped-CG}
Consider applying Algorithm~\ref{alg:capped-CG} with input $U = 0$ to the linear system~\eqref{indef-sys} with $g\neq 0$, $\varepsilon>0$, and $H$ being an $n\times n$ symmetric matrix. Then the number of iterations of Algorithm~\ref{alg:capped-CG} is at most
\[
\min\Big\{n,\Big\lceil \big((\|H\|/\varepsilon)^{1/2}+2\big)\psi(\|H\|/\varepsilon)\Big\rceil\Big\}=\widetilde{\cO}\big(\min\big\{n,(\|H\|/\varepsilon)^{1/2}\big\}\big),
\]
where $\psi(t)=\ln(144((t+2)^{1/2}+1)^2(t+2)^6/\zeta^2)$.
\end{theorem}

\begin{proof}
From \cite[Lemma~1]{RNW18}, we know that the number of iterations of Algorithm \ref{alg:capped-CG} is bounded by $\min\{n,J(U,\varepsilon,\zeta)\}$, where $J(U,\varepsilon,\zeta)$ is the smallest integer $J$ such that $\sqrt{T}\tau^{J/2}\le\hat{\zeta}$, with $U,\hat{\zeta},T$ and $\tau$ being the values returned by Algorithm \ref{alg:capped-CG}. In addition, it was shown in \cite[Section 3.1]{RNW18} that
$J(U,\varepsilon,\zeta)\le \lceil (\sqrt{\kappa}+1/2)\ln(144(\sqrt{\kappa}+1)^2\kappa^6/\zeta^2)\rceil$,
where $\kappa=U/\varepsilon+2$ is an output by Algorithm \ref{alg:capped-CG}. Also, we can observe that $\sqrt{\kappa}\le (U/\varepsilon)^{1/2}+\sqrt{2}\le (U/\varepsilon)^{1/2} +3/2$. Combining these, we obtain that $J(U,\varepsilon,\zeta)\le \lceil[(U/\varepsilon)^{1/2}+2]\ln(144( (U/\varepsilon+2)^{1/2}+1)^2(U/\varepsilon+2)^6/\zeta^2)\rceil$.
Notice from Algorithm~\ref{alg:capped-CG} that the output $U\le\|H\|$. Using these, we obtain the conclusion as desired.
\end{proof}

\section{A randomized Lanczos based minimum eigenvalue oracle} \label{appendix:meo}
In this part we present the randomized Lanczos method proposed in \cite[Section~3.2]{RNW18}, which can be used as a minimum eigenvalue oracle for Algorithms~\ref{alg:NCG-pd} and \ref{alg:NCG}. As briefly discussed in Section~\ref{sec:pd-ncg}, this oracle outputs either a sufficiently negative curvature direction of $H$ or a certificate that $H$ is nearly positive semidefinite with high probability. More detailed motivation and explanation of it can be found in \cite[Section~3.2]{RNW18}.

\begin{algorithm}[h]
	{\small
	\caption{A randomized Lanczos based minimum eigenvalue oracle}
	\label{pro:meo}
		\noindent\textit{Input}: symmetric matrix $H\in\bR^{n\times n}$, tolerance $\varepsilon>0$, and probability parameter $\delta\in(0,1)$.\\
		\noindent\textit{Output:} a sufficiently negative curvature direction $v$ satisfying $v^THv\le-\varepsilon/2$ and $\|v\|=1$; or a certificate that $\lambda_{\min}(H)\ge-\varepsilon$ with probability at least  $1-\delta$.\\
		Apply the Lanczos method \cite{KW92LR} to estimate $\lambda_{\min}(H)$ starting with a random vector uniformly generated on the unit sphere, and run it for at most
		\begin{equation}\label{N-iter}
			N(\varepsilon,\delta):=\min\left\{n,1+\left\lceil\frac{\ln(2.75n/\delta^2)}{2}\sqrt{\frac{\|H\|}{\varepsilon}}\right\rceil\right\}
		\end{equation}
		iterations. If a unit vector $v$ with $v^THv \le -\varepsilon/2$ is found at some iteration, terminate immediately and return $v$.
	}
\end{algorithm}

The following theorem justifies that Algorithm~\ref{pro:meo} is a suitable minimum eigenvalue oracle for Algorithms~\ref{alg:NCG-pd} and \ref{alg:NCG}. Its proof is identical to that of \cite[Lemma~2]{RNW18} and thus omitted.

\begin{theorem}[{\bf iteration complexity of Algorithm~\ref{pro:meo}}]\label{rand-Lanczos}
	Consider Algorithm~\ref{pro:meo} with tolerance $\varepsilon>0$, probability parameter $\delta\in(0,1)$, and symmetric matrix $H\in\bR^{n\times n}$ as its input. Then it either finds a sufficiently negative curvature direction $v$ satisfying $v^THv\le-\varepsilon/2$ and $\|v\|=1$ or certifies that $\lambda_{\min}(H)\ge-\varepsilon$ holds with probability at least  $1-\delta$  in at most $N(\varepsilon,\delta)$ iterations, where $N(\varepsilon,\delta)$ is defined in \eqref{N-iter}.
\end{theorem}

Notice that $\|H\|$ is required in Algorithm~\ref{pro:meo}. In general, computing $\|H\|$  may not be cheap when $n$ is large. Nevertheless, $\|H\|$ can be efficiently estimated via a randomization scheme with high confidence (e.g., see the discussion in \cite[Appendix~B3]{RNW18}).

\end{document}